\documentclass[a4paper,11pt]{amsart}

\usepackage[latin1]{inputenc}

\usepackage{latexsym}
\usepackage{amssymb}
\usepackage{amsmath}
\usepackage{amsthm}
\usepackage[all]{xy}
\usepackage[dvips]{graphicx}

\newtheorem*{teos}{Theorem}

\newtheorem{teo}{Theorem}[section]
\newtheorem{lem}[teo]{Lemma}

\newtheorem{cor}[teo]{Corollary}
\theoremstyle{definition}

\newtheorem{dhef}[teo]{Definition}
\newtheorem{exe}[teo]{Example}
\newtheorem{oss}[teo]{Remark}

\newcommand{\epsi}{\varepsilon}
\newcommand{\lrg}{\longrightarrow}

\newcommand{\C}{\mathbb{C}}

\newcommand{\cepsi}{\mathbb{C}[\varepsilon]}

\newcommand{\shA}{\mathcal{A}}

\newcommand{\shO}{\mathcal{O}}
\newcommand{\shI}{\mathcal{I}}

\newcommand{\de}{\partial}
\newcommand{\debar}{\overline{\partial}}

\newcommand{\Spec}{\operatorname{Spec}}

\newcommand{\MC}{\operatorname{MC}}
\newcommand{\Def}{\operatorname{Def}}

\newcommand{\Hom}{\operatorname{Hom}}

\newcommand{\image}{\operatorname{Im}}
\newcommand{\coker}{\operatorname{coker}}
\newcommand{\ad}{\operatorname{ad}}
\newcommand{\contr}{{\mspace{1mu}\lrcorner\mspace{1.5mu}}}
\newcommand{\bl}{\boldsymbol{l}}

\newcommand{\Art}{\mathbf{Art}}
\newcommand{\Set}{\mathbf{Set}}

\newcommand{\mc}{\operatorname{MC}_{(h,g)}}
\newcommand{\cil}{\operatorname{C}_{(h,g)}}
\newcommand{\hil}{\operatorname{H}_{(h,g)}}
\begin{document}

\title {$L_\infty$-algebras and deformations of holomorphic maps}

\author{Donatella Iacono}

\address{\newline
SISSA-ISAS,\hfill\newline Via Beirut, 2-4, 34014 Trieste\\ Italy.}
\email{iacono@sissa.it}

\begin{abstract}

We construct the deformation functor associated with a pair of
morphisms of differential graded Lie algebras, and use it to study
infinitesimal deformations of   holomorphic maps of compact
complex manifolds. In particular, using $L_\infty$ structures, we
give an explicit description of the differential graded Lie
algebra that controls this problem.
\end{abstract}

\maketitle

\bigskip

\section{Introduction}

The aim of this paper is to develop some algebraic tools to study
infinitesimal deformations of holomorphic maps.

The  modern approach to   deformation  theory is via differential
graded Lie algebras (DGLA for short) or, in general, via
$L_\infty$-algebras.  A DGLA is a differential graded vector space
with a structure of graded Lie algebra, plus some   compatibility
conditions between the differential and the bracket (of the Lie
structure).

Moreover, using the solutions of the Maurer-Cartan equation and
the gauge equivalence, we can associate with a DGLA $L$ a
deformation functor of Artin rings $\Def_L$,  i.e., a functor from
the category $\Art$ of local Artinian $\C$-algebras (with residue
field $\C$) to the category $\Set$ of sets, that satisfies
Schlessinger's conditions $(H_1)$ and $(H_2)$ of
\cite[Theorem~2.1]{bib Artin}.

The guiding principle is the idea  due to P.~Deligne, V.~Drinfeld,
D.~Quillen and M.~Kontsevich (see \cite{bib kontsevich}) that \lq
\lq in characteristic zero  every deformation problem is
controlled by a differential graded Lie algebra".

In other words, we can define a DGLA $L$ (up to quasi-isomorphism)
from the geometrical data of the problem, such that the
deformation functor $\Def_L$ is isomorphic to the deformation
functor of Artin rings that describes the formal deformations of
the geometric object \cite{bib Artin}, \cite{bib Grothendieck} and
\cite{bib Schlessinger}. We point out  that it is easier to study
a deformation functor associated with a DGLA but, in general,  it
is not an easy task to find the right DGLA (up to
quasi-isomorphism) associated with the problem \cite{bib
kontsevich NOTE}.

A first example, in which the associated DGLA is well understood,
is the case of deformations of  complex manifolds. If $X$ is a
complex compact manifold, then its Kodaira-Spencer algebra
controls the infinitesimal deformations of $X$ (Theorem~\ref{teo
def_k =Def_X}).

The next natural problem is to investigate the embedded
deformations of a submanifold in a fixed manifold. Very recently,
M.~Manetti  in \cite{bib ManettiPREPRINT} studies this problem
using the approach via DGLA. More precisely, given a morphism of
DGLAs $ h:L \lrg M$, he describes a general construction to define
a new deformation functor $\Def_h$ associated with $ h$
(Remark~\ref{oss DEf(hg) reduces to DEF_h e DEF L}). Then, by
suitably choosing $L,M$ and $h$ he proves the existence of an
isomorphism between the functor $\Def_h$  and the functor
associated with the infinitesimal deformations of a submanifold in
a fixed manifold.

\smallskip

In this paper, we extend these techniques to study not only the
deformations of an inclusion but, in general, the deformations of
holomorphic maps.

These deformations were first studied from the classical point of
view (no DGLA) by E.~Horikawa  \cite{bib HorikawaI} and \cite{bib
Horikawa III}, then by M.~Namba  \cite{bib Namba},  Z.~Ran
\cite{bib RAn mappe} and, more recently, by E.~Sernesi \cite{bib
Sernesi}.

Roughly speaking, we have a holomorphic map $f:X \lrg Y$ of
compact complex manifolds and we deform both the domain, the
codomain and the map itself. Equivalently, we deform  the graph of
$f$ in the product $X \times Y$,  such that the deformation of
$X\times Y$ is a product of deformations of $X$ and $Y$. Let
$\Def(f)$ be the functor associated with the infinitesimal
deformations of the holomorphic map $f$ (Definition~\ref{def
funore DEF(f)}).

To study these deformations, the key point is the definition  of
the deformation functor $ \Def_{(h,g)} $,  associated with a pair
of morphisms of differential graded Lie algebras $h :L \lrg M$ and
$g: N \lrg M$. In particular, the tangent and obstruction spaces
of $\Def_{(h,g)} $  are the first and second cohomology group of
the suspension of the mapping cone $C^\cdot_{(h,g)}$, associated
with the morphism $h-g:L \oplus N \lrg M$, such that
$(h-g)(l,n)=h(l)-g(n)$ (Section \ref{sez tangent e ostru MC(h,g)
DEF(h,g)}).

By a suitable choice of the morphisms   $h :L \lrg M$ and $g: N
\lrg M $, the functor
 $\Def_{(h,g)}$ encodes all the geometric data of the
problem of infinitesimal deformations of holomorphic maps
(Theorem~\ref{teo Def_(h,g)  Def (f)}).

\begin{teos}[A]
Let $f:X \lrg Y$ be a holomorphic map of compact complex
manifolds. There exist morphisms of DGLAs  $h :L \lrg M$ and $g: N
\lrg M $ such that
$$
\Def_{(h,g)}\cong \Def (f).
$$
\end{teos}

\smallskip

Next, we look for a  DGLA that controls the deformations of
holomorphic maps and, for this purpose, we use $L_{\infty}$
structures.

First, using path objects,  we define a differential graded Lie
algebra $H_{(h,g)}$, for each choice of morphisms $h :L \lrg M$
and $g:N\lrg M$. Then, by transferring $L_{\infty}$ structures, we
explicitly describe an $L_{\infty}$ structure on the cone
$C^\cdot_{(h,g)}$ (Section~\ref{sezio l infinito cono}). In
particular, the functor $\Def_{\hil}$ is isomorphic to  the
deformation functor $\Def^\infty_{\cil^\cdot} $ associated with
this $L_{\infty}$ structure on   $C^\cdot_{(h,g)}$
(Corollary~\ref{cor DEF inf cono iso DEF hil}).

Finally, we prove that the  deformation functor
$\Def^\infty_{\cil^\cdot} $  coincides with the deformation
functor $\Def_{(h,g)}$ associated with the pair $(h,g)$
(Theorem~\ref{teo Def INFI cil  = Def(h,g)}) and so
$\Def_{\hil}\cong \Def_{(h,g)}$ (Corollary~\ref{cor L infini
Def_(h,g)iso Def_H(h,g)}).

Therefore, in particular, we give  an explicit description (more
than the existence) of a DGLA that controls the deformations of
holomorphic maps (Theorem~\ref{teo esiste hil governa DEF(f)}).

\begin{teos}[B]
Let $f:X \lrg Y$ be a holomorphic map of compact complex manifold.
Then, there exists an explicit description of a DGLA $H_{(h,g)}$
such that
$$
\Def_{H_{(h,g)}}\cong \Def (f).
$$
\end{teos}

\medskip

When we developed the techniques of this paper, we had also in
mind some applications to the study of obstruction theory.
However, since the number of pages grew, we decided to split the
material, collecting here the general theory and leaving for the
sequel \cite{bib articolo2 SEMIRe} (in preparation) the study of
obstructions and of semi-regularity maps that annihilates
obstructions.

\smallskip

{ \bf Acknowledgments.} It is a pleasure for me to show my deep
gratitude to the advisor of my PhD thesis Prof. Marco Manetti, an
excellent   helpful professor who supports and encourages me every
time. I'm indebted with him for many useful discussions, advices
and suggestions. Several ideas of this work are grown under his
influence. This paper was written at the Mittag-Leffler Institute
(Djursholm, Sweden),  during  my participation in the \emph{Moduli
Spaces} Program, Spring 2007. I am very grateful for the
hospitality and support received. I am also grateful to the referee for improvements in the presentation of the paper.

\section{Notation}

We will work over the field  $\C$ of complex numbers, although
most of the algebraic results are valid over an arbitrary field of
characteristic zero. All vector spaces, linear maps, tensor
products etc. are intended over $\C$.

Unless otherwise specified,  any (complex) manifold is assumed compact and connected.

Given a manifold $X$, we denote by $\Theta_X $ the holomorphic
tangent bundle, by $\shA_X^{p,q}$ the sheaf of differentiable
$(p,q)$-forms on $X$ and by $A_X^{p,q}=\Gamma(X,\shA_X^{p,q})$ the
vector space of   global sections of  $\shA_X^{p,q}$. More
generally, $\shA_X^{p,q}(\Theta_X)$ is the sheaf of differentiable
$(p,q)$-forms on $X$ with values in  $\Theta_X $ and
$A_X^{p,q}(\Theta_X)=\Gamma(X,\shA_X^{p,q}(\Theta_X))$ is the
vector space of  its global sections.

Let $f:X \lrg Y$ be a  holomorphic map of 
manifolds. We denote by  $f^*$ and $f_*$   the map induced by $f$,
i.e.,
$$
f^*: A_Y^{p,q}(\Theta_Y) \lrg A_X^{p,q}(f^*\Theta_Y) \quad \mbox{
and} \quad f_*: A_X^{p,q}(\Theta_X) \lrg A_X^{p,q}(f^*\Theta_Y).
$$

\section{Background}

Let $L=(\oplus _i L_i,d,[\, ,\,])$ be a DGLA and $(A,m_A) \in \Art
$, where $m_A$ denotes the maximal ideal of $A$. The set of
Maurer-Cartan elements with coefficients in $A$ is defined as
follows
$$
\MC_L(A)=\{ x \in L^1\otimes m_A \ |\ dx+ \displaystyle\frac{1}{2}
[x,x]=0 \},
$$
where the DGLA structure on $L \otimes m_A$ is the natural
extension of the DGLA structure on $L$. For each $a \in L^0\otimes
m_A$, we define the gauge action $*:\exp(L^0 \otimes m_A)\times
\MC_L(A)\lrg\MC_L(A)$ by the formula
$$
e^a * x:=x+\sum_{n\geq 0}  \frac{ [a,-]^n}{(n+1)!}([a,x]-da).
$$
Given $x \in \MC_L (A) $,  the \emph{irrelevant stabilizer} $
Stab_A(x)$ of $x$ is by definition
$$
Stab_A(x)=\{e^{dh+[x,h]}|\, h \in L^{-1}\otimes m_A\}.
$$
The set $ Stab_A(x)$ is a subgroup of $\exp(L^0\otimes A )$, that
is contained in the stabilizer of $x$ and it satisfies the
following property:
$$
\forall\  a \in L^0\otimes A  \qquad e^a
Stab_A(x)e^{-a}=Stab_A(y), \quad \mbox{ with } \quad y=e^a*x.
$$
The deformation functor $ \Def_L:\Art \lrg \Set$ associated with a
DGLA $L$ is:
$$
\Def_L(A)=\frac{\{ x \in L^1\otimes m_A \ |\ dx+
\displaystyle\frac{1}{2} [x,x]=0 \}}{  \exp(L^0\otimes A )  }.
$$

\begin{dhef}
A functor of Artin rings $F:\Art \lrg \Set $ is  \emph{controlled}
by a DGLA $L$ if $F $ is isomorphic to $\Def_L$.
\end{dhef}

\begin{exe}
Let $X$ be a manifold. The  \emph{Kodaira-Spencer} (differential
graded Lie) algebra of $X$ is
$$
KS_X=\bigoplus_i \Gamma(X,\shA_X^{0,i}(\Theta_X))=\bigoplus_i
A_X^{0,i}(\Theta_X).
$$
The differential $\tilde{d}$   is the opposite of the Dolbeault
differential, whereas the bracket $[\ , \ ]$ is defined in local
coordinates as the $\overline{\Omega}^*$-bilinear extension of the
standard bracket on $\shA^{0,0}_X(\Theta_X)$
($\overline{\Omega}^*=\ker (\de : \shA_X^{0,*} \lrg \shA_X^{1,*})$
is the sheaf of anti-holomorphic differential forms). Explicitly,
if $z_1, \ldots,z_n$ are local holomorphic coordinates on $X$, we
have
$$
\tilde {d}(f d\overline{z}_I \frac{\de}{\de z_i})=-\debar
(f)\wedge d\overline{z}_I \frac{\de}{\de z_i},
$$
$$
[f\frac{\de}{\de z_i}\, d\overline{z}_I,g\frac{\de}{\de z_j} \,
d\overline{z}_J]=(f\frac{\de g}{\de z_i}\frac{\de}{\de z_j}-
g\frac{\de f}{\de z_j}\frac{\de}{\de z_i})\, d\overline{z}_I
\wedge d \overline{z}_J ,\qquad \forall  \ f,g \in \shA_X^{0,0}.
$$
Then, $\shA_X^{0,*}(\Theta_X)$ is a sheaf of DGLAs.

\medskip

Define the \emph{holomorphic Lie derivative}
$$
\bl: \shA^{0,*}_X(\Theta_X)\lrg Der^*(\shA_X^{*,*}),
$$
$$
\bl_a(\omega)=\de(a \contr \omega)+ (-1)^{\deg(a)} a \contr \de
\omega,
$$
for each $a \in \shA^{0,*}_X(\Theta_X)$ and $ \omega \in
\shA_X^{*,*}$.

The DGLA sheaf morphism   $ \bl$ is    injective; moreover, using
$ \bl$, we define, for any object $(A,m_A) \in \Art$ and   $a \in
\shA_X^{0,0}(\Theta_X) \otimes m_A$, the automorphism $e^a$ of
$\shA _X^{0,*} \otimes A$:
\begin{equation}\label{equa def di e^a su shA^(0,*)tensor A}
e^a : \shA _X^{0,*} \otimes A \lrg \shA _X^{0,*} \otimes A, \ \ \
\ \ f \longmapsto e^a(f)=\sum^\infty_{n=0}\frac{\bl_a^n}{n!}(f).
\end{equation}

\begin{lem}\label{lemma e^a(d+l_x)e^-a = d+e^a *x}
For every local Artinian $\C$-algebra $(A,m_A)$,   $a\in
\shA_X^{0,0}(\Theta_X)\otimes m_A$ and  $x
 \in \MC_{KS_X}(A)$  we have
\begin{equation}\label{equa COMM e^a o d +x o e^-a=e^a*x}
e^a\circ (\debar +\bl_x) \circ e^{-a}=\debar + e^a*\bl_x\ :\
\shA_X^{0,0} \otimes A \lrg \shA_X^{0,1} \otimes A,
\end{equation}
where $*$ is the gauge action. In particular,
$$
\ker(\debar + e^a * \bl_x: \shA_X^{0,0}\otimes A \lrg
\shA_X^{0,1}\otimes A)=e^a(\ker(\debar +  \bl_x:
\shA_X^{0,0}\otimes A \lrg \shA_X^{0,1}\otimes A)).
$$

\end{lem}
\begin{proof}
See \cite[Lemma~5.1]{bib ManettiPREPRINT} or
\cite[Lemma~II.5.5]{bib tesidottorato}.

\end{proof}

\noindent Let  $\Def_X\colon\mathbf{Art}\to\mathbf{Set}$ be the
functor of infinitesimal deformations of $X$, i.e.,
$$
\Def_X(A)=\frac{\{\text{deformations $X_A$  of $X$ over
$\Spec(A)$}\}}{\sim}.
$$
Recall that a deformation  $X_A$  of $X$ over $\Spec(A)$ is
nothing else that a morphism $\mathcal{O}_{X_A}\to \mathcal{O}_X$
of sheaves of $A$-algebras such that $\mathcal{O}_{X_A}$ is flat
over $A$ and the induced map
$\mathcal{O}_{X_A}\otimes_A\mathbb{C}\to \mathcal{O}_X$ is an
isomorphism. Moreover,  $\Def_X$ has $H^1(X,\Theta_X)$ and
$H^2(X,\Theta_X)$ as   tangent  and obstruction space,
respectively.

The following theorem is well known and a proof based on the
theorem of Newlander-Nirenberg can be found in \cite{bib
Catanesecime}, \cite{bib GoMil III}  or more recently in \cite{bib
manRENDICONTi}. For a proof that avoid this theorem see
\cite[Theorem~II.7.3]{bib tesidottorato}.

\begin{teo} \label{teo def_k =Def_X}
Let $X$ be a  manifold and ${KS_X}$ its
Kodaira-Spencer algebra. Then there exists an isomorphism of
functors
$$
 \gamma': \Def_{KS_X} \lrg \Def_X,
$$
defined in the following way: given a local Artinian $\C$-algebra
$(A,m_A)$ and a solution of the Maurer-Cartan equation $x \in
A^{0,1}_X(\Theta_X) \otimes m_A$, we set
$$
\shO_{X_A}(x)=\ker(\shA^{0,0}_X\otimes
A\xrightarrow{\debar+\bl_x}\shA^{0,1}_X\otimes A),
$$
and the map $\shO_{X_A}(x) \lrg \shO_X$ is induced by the
projection $\shA^{0,0}_X \otimes A \lrg \shA^{0,0}_X \otimes
\C=\shA^{0,0}_X$.
\end{teo}

\end{exe}

\section{Deformation functor of a pair of morphisms of
DGLAs}\label{sezio def funct of pair morfisms}

Let $h:L \lrg M$ be a morphism of DGLAs. \emph{The suspension of
the mapping cone of $h$} is   the complex $(C^\cdot_h,\delta)$,
where $ C_h^i=L^i   \oplus M^{i-1} $ and $\delta (l, m)= ( d l,
h(l) -d m ) $.

Let $h:(L,d ) \lrg (M,d )$ and $g:(N,d ) \lrg (M,d )$ be morphisms
of DGLAs:
\begin{center}
$\xymatrix{  & L \ar[d]^h\\
N\ar[r]^g & M. \\
}$
\end{center}

The \emph{suspension of the mapping cone of the pair $(h,g)$ } is
the differential graded vector space $(\cil^\cdot,D)$, where
$$
\cil^i=L^i \oplus N^i \oplus M^{i-1}
$$
and the differential $D $ is defined as follows
$$
L^i \oplus N^i \oplus M^{i-1} \ni (l,n,m) \stackrel{D}{\lrg} ( d
l, d n,-d m-g(n)+h(l)) \in L^{i+1} \oplus N^{i+1} \oplus M^i.
$$

The projection $\cil^\cdot \lrg L^\cdot \oplus N^\cdot$ is a
morphism of complexes and so there exists the following exact
sequence
\begin{center}
$ \xymatrix{0 \ar[r] & (M^{\cdot-1},-d ) \ar[r]   &
(C^\cdot_{(h,g)},D) \ar[r]   & (L^\cdot \oplus N^\cdot,d) \ar[r] &
0
\\ }$
\end{center}
that induces
\begin{equation}\label{success esatta lunga omolog cil(fg)}
\cdots \lrg H^i(\cil^\cdot) \lrg H^i(L^\cdot\oplus N^\cdot) \lrg
H^i(M^\cdot) \lrg H^{i+1}(\cil^\cdot) \lrg \cdots .
\end{equation}

Note that, in general, we can not define any bracket on the cone
$\cil^\cdot$, such that $\cil^\cdot$ is a DGLA and the projection
$\cil^\cdot\lrg L \oplus N$ is a morphism of DGLAs. In Section
\ref{sezio l infinito cono}, we will define an $L_\infty$
structure on $\cil^\cdot$.

\begin{lem}\label{lem h inj then C_(h,g)=C_pi dopo g}
Let $g:N \lrg M$ and $h:L \lrg M$ be morphisms of complexes with
$h $ injective, i.e., there exists the exact sequence of complexes
\begin{center}
$\xymatrix{0 \ar[r] &L  \ar[r]^h &  M \ar[r]^{\pi\ } & \coker
(h)\ar[r] & 0 \\
& & N.\ar[u]_g & & \\
}$
\end{center}
Then, $(\cil^\cdot,D)$ is quasi-isomorphic to $(C^\cdot_{\pi \circ
g},\delta)$.

\end{lem}
\begin{proof}

Let $\gamma:\cil^\cdot \lrg C^\cdot_{ \pi \circ g }$ be defined as
$$
\cil^i \ni (l,n,m) \stackrel{\gamma}{\longmapsto} (-n,\pi(m)) \in
C^i_{ \pi \circ g }.
$$
Then, a straightforward computation shows that $\gamma$ is a
quasi-isomorphism.

\end{proof}

\subsection{The functors $\MC_{(h,g)}$ and $\Def_{(h,g)}$ }
\label{sezi def MC_(h,g) and DEF(h,g) no ext}

The \emph{Maurer-Cartan functor associated with the pair $(h,g)$ }
is defined as follows
$$
\MC_{(h,g)}: \Art \lrg \Set,
$$
$$
\MC_{(h,g)}(A)=\{(x,y,e^p) \in (L^1 \otimes m_ A)\times  (N^1
\otimes m_A ) \times \operatorname{exp}(M^0 \otimes m_A)  |
$$
$$
dx +\frac{1}{2}[x,x]=0,\   dy +\frac{1}{2}[y,y]=0,\  g(y)=e^p*h(x)
\}.
$$

We note that $\MC_{(h,g)} $ is  a \emph{homogeneous} functor,
i.e., for each $B \lrg A$ and $C \lrg A$ in $\Art$, $\MC_{(h,g)}(B
\times _A C) \cong \MC_{(h,g)}(B)\times
_{\MC_{(h,g)}(A)}\MC_{(h,g)}(C)$.

\begin{oss}
In \cite[Section~2]{bib ManettiPREPRINT}, M.~Manetti introduced
the functor $\MC_h:\Art \lrg \Set$, associated with a morphism
$h:L \lrg M$ of DGLAs, by defining, for each $(A,m_A) \in \Art$,
$$
\MC_{h}(A)=
$$
$$
\{(x, e^p) \in (L^1 \otimes m_ A)  \times \operatorname{exp}(M^0
\otimes m_A)  |\  dx +\frac{1}{2}[x,x]=0,\, e^p*h(x)=0 \}.
$$
Therefore, if we take $N=0$ and $g=0$, the new functor
$\MC_{(h,g)}$ reduces to the old one $\MC_h$. By choosing $M=N=0$
and $h=g=0$, $\MC_{(h,g)}$ reduces to the Maurer-Cartan functor
$\MC_L$ associated with the DGLA $L$.
\end{oss}

Next, we consider on $\MC_{(h,g)}(A)$ the following relation $
\sim$:
$$
(x_1,y_1,e^{p_1})\, \sim\,(x_2,y_2,e^{p_2})
$$
if and only if there exist $a \in L^0 \otimes m_A$, $b \in N^0
\otimes m_A $ and $c \in M^{-1}\otimes m_A $ such that
$$
x_2=e^a *x_1, \qquad y_2=e^b*y_1
$$
and
$$
e^{p_2}=e^{g(b)}e^T e^{p_1} e^{-h(a)}, \quad \mbox{with }\quad
T=dc+[g(y_1),c].
$$
By definition of the irrelevant stabilizer, we note  that $e^T\in
Stab_A(g(y_1))$. An easy computation shows that   $ \sim$ is a
well defined equivalence relation \cite[Lemma~III.2.23]{bib
tesidottorato}. Then, it makes sense to consider the following
functor.

\begin{dhef}\label{defin funtor Def_(h,g) non esteso}
The \emph{deformation functor  associated with a  pair $(h,g)$} of
morphisms of  differential graded Lie algebras is:
$$
\Def_{(h,g)}:\Art \lrg \Set,
$$
$$
\Def_{(h,g)}(A)= \frac{\MC_{(h,g)}(A)}{ \sim }.
$$
\end{dhef}

\begin{oss}\label{oss DEf(hg) reduces to DEF_h e DEF L}
In \cite[Section~2]{bib ManettiPREPRINT}, M.~Manetti defined the
functor $\Def_h$ associated with a morphism $h:L \lrg M$ of DGLAs:
$$
\Def_h: \Art \lrg \Set,
$$
$$
\Def_h(A)= \frac{\MC_{h}(A)}{ \exp(L^0 \otimes m_A) \times
\exp(dM^{-1} \otimes m_A)},
$$
where  the gauge action of $\exp(L^0 \otimes m_A) \times
\exp(dM^{-1} \otimes m_A)$ is given by the   formula
$$
(e^a,e^{dm})*(x,e^p)= (e^a*x,e^{dm} e^pe^{-h(a)}),\qquad \forall \
a \in  L^0 \otimes m_A, m \in  M^{-1} \otimes m_A.
$$

Therefore, if we take $N=0$ and $g=0$, the new functor
$\Def_{(h,g)} $ reduces to the old one $\Def_h$.

By choosing $N=M=0$ and $h=g=0$, $\Def_{(h,g)}$ reduces to the
Maurer-Cartan functor  $\Def_L$ associated with the DGLA $L$.
\end{oss}

\begin{oss}\label{oss proj varrho induce no EXT Def(h,g)->Def_N}
Consider the functor $ \Def_{(h,g)}$. Then the projection
$\varrho$ on the second factor:
$$
\varrho: \Def_{(h,g)} \lrg \Def_N,
$$
$$
\Def_{(h,g)}(A) \ni (x,y,e^p)\stackrel{\varrho}{\lrg} y \in
\Def_N(A)
$$
is a morphism of deformation functors.
\end{oss}
\begin{oss}\label{oss DEF_(h,g) con h iniettivo}
If the morphism \emph{$h$ is injective}, then  for each $(A,m_A)
\in \Art$  the functor $\MC_{(h,g)}$ has the following form:
$$
\MC_{(h,g)}(A)=\{(x,e^p) \in  (N^1 \otimes m_A)\times
\operatorname{exp}(M^0 \otimes m_A)  |
$$
$$
dx +\frac{1}{2}[x,x]=0,\  e^{-p}*g(x) \in L^1 \otimes m_A \}.
$$
If \emph{ $M$ is also concentrated in non negative degrees}, then
the gauge equivalence   is given by
$$
(x,e^p)\sim (e^b*x,e^{g(b)} e^p e^{a}), \qquad \mbox{ with } a \in
L^0 \otimes m_A \mbox{ and } b \in N^0 \otimes m_A.
$$

\end{oss}

\subsection{Tangent and obstruction spaces of $\MC_{(h,g)}$
and $\Def_{(h,g)}$}\label{sez tangent e ostru MC(h,g) DEF(h,g)}

By definition, the tangent space of a functor of Artin rings $F$
is $F(\C[\epsi])$, where $\epsi^2=0$. Therefore,
$$
\MC_{(h,g)}(\cepsi)=
$$
$$
=\{(x,y,e^p) \in (L^1 \otimes \C\epsi) \times  (N^1 \otimes
\C\epsi) \times  \operatorname{exp}(M^0 \otimes \C\epsi)|
$$
$$
dx=dy=0, h(x)-g(y)-dp=0 \}
$$
$$
\cong \{(x,y, p) \in L^1 \times  N^1  \times
 M^0  |\ dx=dy=0, -dp -g(y) +h(x)=0\}=
$$
$$
\ker (D:\cil^1 \lrg \cil^2),
$$
and
$$
\Def_{(h,g)}(\cepsi)\cong
$$
$$
\frac{\{(x,y, p) \in L^1 \times  N^1  \times
 M^0  |\ dx=dy=0, g(y)=h(x)-dp \} }{
\{ (-da,-db,dc+g(b)-h(a))|\ a \in L^0  , b \in N^0
 , c  \in M^{-1}   \}}
$$
$$
\cong H^1(\cil^\cdot ).
$$

The obstruction space of $\Def_{(h,g)}$, is naturally contained in
$H^2 (\cil^\cdot)$. Indeed, let
$$
0 \lrg J \lrg {\tilde{A}} \stackrel{\alpha}{\lrg} A\lrg 0
$$
be a small extension and $(x,y,e^p) \in \MC_{(h,g)}(A)$.

Since $\alpha$ is surjective, there exist  $\tilde{x} \in L^1
\otimes m_{\tilde{A}} $ that lifts $x$, $\tilde{y}\in N^1 \otimes
m_{\tilde{A}} $ that lifts $y$, and  $q \in M^0 \otimes
m_{\tilde{A}}$ that lifts $p$. Let
$$
l=d\tilde{x}+ \displaystyle\frac{1}{2}[\tilde{x},\tilde{x}] \in
L^2 \otimes m_{\tilde{A}}
$$
and
$$
k=d\tilde{y}+ \displaystyle\frac{1}{2}[\tilde{y},\tilde{y}] \in
N^2 \otimes m_{\tilde{A}}.
$$
It is easy to see that $\alpha(l)=\alpha(k)=dl=dk=0$; then $l \in
H^2(L) \otimes J$ and $k \in H^2(N) \otimes J$.

Let $r =-g(\tilde{y})+e^q * h(\tilde{x})\in M^1\otimes
m_{\tilde{A}}$; thus, $\alpha(r)=0$ or, equivalently,  $r \in
M^1\otimes J$. It can be proved that $-dr -g(k)+h(l)=0$ and so
$(l,k,r) \in Z^2(\cil ^\cdot)\otimes J$. Let $[(l,k,r)] $ be the
class in $H^2(\cil ^\cdot)\otimes J$. This class does not depend
on the choice of the liftings and  it vanishes  if and only if
there exists a lifting of $(x,y,e^p) \in \MC_{(h,g)}(A)$ in $
\MC_{(h,g)}({\tilde{A}})$  (\cite[Lemma~III.1.19]{bib
tesidottorato}).

\section{Deformations  of holomorphic maps}

\begin{dhef}
Let $f:X \lrg Y$ be a  holomorphic map of 
manifolds and $A \in \Art$. An \emph{infinitesimal deformation of
f  over $\Spec(A)$} is a commutative diagram of complex spaces
\begin{center}
$\xymatrix{X_A\ar[rr]^{\mathcal{F}} \ar[dr]_\pi &  & Y_A
\ar[ld]^\mu \\
            & \Spec(A), &  \\ }$
\end{center}
where   $(X_A,\pi,\Spec(A))$ and $(Y_A,\mu,\Spec(A))$ are
infinitesimal deformations of $X$ and $Y$, respectively, and
$\mathcal{F}$ is a holomorphic map that restricted to the fibers
over the closed point of $\Spec(A)$ coincides with $f$. If
$A=\C[\epsi]$ we have a \emph{first order deformation} of $f$.
\end{dhef}

\begin{dhef}

Let
\begin{center}
$\xymatrix{X_A\ar[rr]^{\mathcal{F}} \ar[dr]_\pi &  & Y_A
\ar[ld]^\mu &   \mbox{ and }    &X'_A\ar[rr]^{\mathcal{F}'}
\ar[dr]_{\pi'} & &
Y'_A \ar[ld]^{\mu'}\\
            & \Spec(A) & &        &  &\Spec(A) \\ }$
\end{center}
be two infinitesimal deformations of $f$ over $\Spec(A)$. They are
\emph{isomorphic} if there exist bi-holomorphic maps $\phi : X_A
\lrg X'_A$ and $\psi: Y_A\lrg Y'_A$ (that are equivalences of
infinitesimal deformations of $X$ and $Y$, respectively) such that
the following diagram is commutative:
\begin{center}
$\xymatrix{X_A\ar[rr]^{\mathcal{F}} \ar[d]_\phi &  & Y_A
\ar[d]^\psi \\
X'_A\ar[rr]^{\mathcal{F}'} &  & Y'_A. \\ }$
\end{center}
\end{dhef}

\begin{dhef}\label{def funore DEF(f)}
The \emph{functor of infinitesimal deformations} of a  holomorphic
map $f:X \lrg Y$ is
$$
\Def(f): \Art \lrg \Set,
$$
$$
\ \ \ \ \ \  \ \ \ \  \ \ A\longmapsto \Def(f)
(A)=\left\{\begin{array}{c} \mbox{ isomorphism  classes  of}\\
 \mbox{ infinitesimal   deformations     }\\
 \mbox {of } f  \mbox { over }  \Spec(A) \\
\end{array} \right\}.
$$
\end{dhef}

\begin{oss}\label{oss defo mapp=Def grafo prod defX.defY}
Let $\Gamma$ be the graph of $f$ in the product $X \times Y$. The
infinitesimal deformations of $f$ can be interpreted as
infinitesimal deformations $ \Gamma_A $ of $\Gamma$  in the
product $X \times Y$, such that the induced deformations $ (X
\times Y)_A$ of $X \times Y$ are products of infinitesimal
deformations of $X$ and of $Y$. Since not all the deformations of
a product are products of deformations (\cite[pag.~436]{bib
Kodaira SpencerII}), we are not just considering the deformations
of the graph in the product. Moreover, with this interpretation,
two infinitesimal deformations $\Gamma_A \subset (X \times Y)_A $
and $\Gamma_A' \subset (X \times Y)_A'$ are equivalent if there
exists an isomorphism $\phi:(X \times Y)_A\lrg (X \times Y)_A'$ of
infinitesimal deformations of $X\times Y$ such that
$\phi(\Gamma_A)=\Gamma_A'$.

\end{oss}

Let $(B^\cdot,D_{\debar})$  be the complex with
$$
B^p=A_X^{(0,p)}( \Theta_X) \oplus A_Y^{(0,p)}(\Theta_Y) \oplus
  A_X^{(0,p-1)}( f^*\Theta_Y)
$$
and
$$
D_{\debar}: B^p \lrg B^{p+1}, \qquad (x,y,z)\longmapsto(\debar
x,\debar y, \debar z+ (-1)^p(f_*x-f^*y)).
$$

\begin{teo}[E.~Horikawa]\label{teo TANG e OSTR DEF(f) horikawa}
$H^1 (B^\cdot)$ is in one-to-one  correspondence with the first
order deformations of $f:X \lrg Y$.

The obstruction space of the functor $\Def(f)$ is naturally
contained in $H^2 (B^\cdot)$.
\end{teo}

\begin{proof}
See \cite[Section~3.6]{bib Namba}.

\end{proof}

\begin{oss}\label{oss class H^2(TX,TY,f*TY) def by  Xe,Ye}
Consider a first order deformation $f_\epsi$ of $f$: in
particular, we are considering   first order deformations
$X_\epsi$ and $Y_\epsi$, of $X$ and  $Y$, respectively.

Then, we associate with $X_\epsi$ a class $x \in H^1(X,\Theta_X)$
and with $Y_\epsi$ a class $y \in H^1(Y,\Theta_Y)$. Therefore, the
class in $H^1 (B^\cdot)$ associated with $f_\epsi$ is $[(x,y,z)]$,
with $z \in  A_X^{(0,0)} ( f^*\Theta_Y))$ such that $\debar z =
f_*x-f^*y$.

Analogously, let  $0\lrg J \lrg {\tilde{A}} \lrg A\lrg 0$ be a
small extension and $\mathcal{F}_A$ an infinitesimal deformation
of $f$ over $\Spec (A)$. If $h \in H^2(X,\Theta_X)$ and $k \in
H^2(Y, \Theta_Y)$ are the obstruction classes associated with
$X_A$ and $Y_A$, respectively, then the obstruction class in $H^2
(B^\cdot)$ associated with $\mathcal{F}_A$ is $[(h,k,r)]$, with $r
\in A_X^{(0,1)} ( f^*\Theta_Y))$ such that $\debar r =
-(f_*h-f^*k)$.
\end{oss}

\bigskip

Let $Z:=X \times Y $ be the product of $X$ and $Y$ and $p:Z\lrg X$
and $q:Z \lrg Y$  the natural projections. Defining the morphism
$$
F:X \lrg \Gamma \subseteq Z,
$$
$$
\ \ \ x \longmapsto (x,f(x)),
$$
we get the following commutative diagram:
\begin{center}
$\xymatrix{X \ar[rrrr]^{F} \ar[ddrrr]_{id} \ar[ddrrrrr]^f & &
& & Z\ar[ddl]_{ p} \ar[ddr]_{ \ \ q} & \\
& & &  & &  \\
& & & X & & Y. \\ }$
\end{center}
In particular, $F^*\circ p^*=id$ and $F^* \circ q^*=f^*$. Since
$\Theta_Z=p^*\Theta_X \oplus q^*\Theta_Y$, it follows that
$F^*(\Theta_{Z})=\Theta_X \oplus f^*\Theta_Y$. Define the morphism
$\gamma:\Theta_Z \lrg f^*\Theta_Y$ as the product
$$
\gamma: \Theta_Z \stackrel{F^*}{\lrg}\Theta_X \oplus f^* \Theta_Y
\stackrel{(f^*,-id)}{\lrg} f^*\Theta_Y;
$$
moreover, let $\pi$ be the following surjective morphism:
$$
A^{0,*}_{Z}(\Theta_{Z}) \stackrel{\pi}{\lrg}
A_X^{0,*}(f^*\Theta_Y) \lrg 0,
$$
$$
\pi(\omega\, u)= F^*(\omega)\gamma(u), \qquad \forall \ \omega \in
A^{0,*}_{Z}, \ u \in \Theta_{Z}.
$$
Since each $u \in \Theta_{Z}$ can be written as $u=p^*v_1+q^*v_2$,
for some $v_1 \in\Theta_X$ and $v_2 \in \Theta_Y$, we also have
$$
\pi (\omega u)=F^*(\omega)(f_*(v_1)-f^*(v_2)).
$$
Since $F^*\debar=\debar F^*$, $\pi$ is a morphism of complexes.

Let $\mathcal{L}$ be the kernel of $\pi$:
\begin{equation}\label{equa def L con f^*T_X}
0 \lrg \mathcal{L}\stackrel{h}{ \lrg} \shA^{0,*}_{Z}(\Theta_{Z})
\stackrel{\pi}{\lrg} \shA_X^{0,*}(f^*\Theta_Y) \lrg 0
\end{equation}
and $h:\mathcal{L}\lrg \shA^{0,*}_{Z}(\Theta_{Z}) $ the inclusion.

Since there is a canonical  isomorphism between  the normal bundle
$N_{\Gamma|Z}$ of $\Gamma$ in $Z$  and the pull-back $f^*T_Y$,
(\ref{equa def L con f^*T_X}) reduces to
$$
0 \lrg \mathcal{L}\stackrel{h}{ \lrg} \shA^{0,*}_{Z}(\Theta_{Z})
\stackrel{\pi}{\lrg} \shA_\Gamma^{0,*}(N_{\Gamma|Z}) \lrg 0.
$$

Let $i:\Gamma \lrg Z$ be the inclusion and $i^*:\shA^{0,*}_{Z}
\lrg \shA^{0,*}_{\Gamma} $ the induced map. Suppose that
$z_1,\ldots,z_n$ are holomorphic coordinates on $Z$ such that
$Z\supset \Gamma=\{z_{t+1}=\cdots=z_n=0\}$. Then, $\displaystyle
\omega=\sum_{j=1}^n \omega_j \frac{\de}{\de z_j} \in
\shA^{0,*}_{Z}(\Theta_{Z})$ lies in $\mathcal{L }$   if and only
if  $\omega_j \in\ker i^*$ for $j \geq t+1$. In particular,
$\mathcal{L }^0$ is the sheaf of differentiable vector fields on
$Z$ that are tangent to $\Gamma$.

\begin{lem}\label{lemm L manetti is DGLA}
$\mathcal{L }$ is a sheaf of differential graded Lie subalgebras
of $\shA_Z^{0,*}(\Theta_Z)$ such that $\bl_a (\ker i^*)\subset
\ker i^*$ if and only if $a \in \mathcal{L }\subset
\shA_Z^{0,*}(\Theta_Z)$. Moreover,  consider the  automorphism
$e^a$ of $\shA _Z^{0,*} \otimes A$ defined in (\ref{equa def di
e^a su shA^(0,*)tensor A}): if $a \in  \mathcal{L }^0\otimes m_A$
then $e^a(\ker (i^*)\otimes A)=\ker (i^*) \otimes A$.
\end{lem}
\begin{proof}
See \cite[Section~5]{bib ManettiPREPRINT}. It is an easy
calculation in local holomorphic coordinates.

\end{proof}

Let $L$ be the differential graded  Lie algebra of global sections
of $\mathcal{L}$.

\smallskip

Let $M$ be the Kodaira-Spencer algebra of the product $Z$, i.e.
$M=KS_Z$, and $h:L \lrg M$ be the inclusion.
\smallskip

Let $N=KS_X \times KS_Y$ be the product of the Kodaira-Spencer
algebras of $X$ and of $Y$  and  $g=p^*+q^*:KS_X \times KS_Y\lrg
KS_{Z}$, i.e., $g(n_1,n_2)=p^*n_1 +q^*n_2$ (for $n=(n_1,n_2)$ we
also use the notation $g(n)$).

Therefore, we get the diagram
\begin{center}
\begin{equation}\label{equa diagrL-> KS(XxY) <-KS(X)xKS(Y)}
\xymatrix{ & & L \ar@{^{(}->}[d]^h   \\
N=KS_X \times KS_Y \ar[rr]^{g=(p^*,q^*)} &  & M= KS_{Z}.
\\ }
\end{equation}
\end{center}

\begin{oss}

Given morphisms of DGLAs $h:L \lrg {KS} _{Z}$ and $g:KS_X\times
KS_Y \lrg {KS} _{Z}$, we can consider the complex
$(\cil^\cdot,D)$, with $C^i_{(h,g)}=L^i \oplus KS_X^i\oplus KS_Y^i
\oplus {KS}^{i-1}_{Z}$ and differential is given by $D(l,n_1,n_2,m)=(-\debar
l, -\debar n_1, -\debar n_2, \debar m -p^*n_1 -q^* n_2 + h(l))$.

Using the morphism $\pi : KS_{X\times Y} \lrg
A_X^{0,*}(f^*\Theta_Y)$, we can define a morphism
$$
\beta :(\cil^\cdot, D) \lrg (B^\cdot, D_{\debar}),
$$
$$
\beta (l,n_1,n_2,m)=((-1)^i n_1,(-1)^i n_2,-\pi (m)) \qquad
\forall \ (l,n_1,n_2,m)\in \cil^i.
$$

\begin{lem}\label{lem isomor complesso Namba con Cono}
$\beta:(\cil^\cdot, D) \lrg (B^\cdot, D_{\debar}) $ is a morphism
of complexes which is a quasi-isomorphism.

\end{lem}
\begin{proof}
It follows from an easy computation.

\end{proof}
\end{oss}

Let us consider the functor $\Def_{(h,g)}$ associated with diagram
(\ref{equa diagrL-> KS(XxY) <-KS(X)xKS(Y)}).  Since $h$ is
injective and $M$ is concentrated in non negative degrees, by
Remark~\ref{oss DEF_(h,g) con h iniettivo}, for each $(A,m_A) \in
\Art$, we have
$$
\Def_{(h,g)}(A)=\{(n,e^m) \in  (N^1 \otimes m_A)\times
\operatorname{exp}(M^0 \otimes m_A)  |
$$
$$
dn +\frac{1}{2}[n,n]=0,\  e^{-m}*g(n) \in L^1 \otimes m_A \}/\sim,
$$
where $(x,e^p)\sim (e^b*x,e^{g(b)} e^p e^{a})$, with $ a \in L^0
\otimes m_A$ and $b \in N^0 \otimes m_A$.

\begin{oss}\label{oss (n,e^m) DEF(h,g)-> def T in def Xx defY}
Let $(n,e^m) \in \Def_{(h,g)}$. In particular, $n=(n_1,n_2)$
satisfies the Maurer-Cartan equation and so $n_1 \in \MC_{KS_X}$
and $n_2 \in \MC_{KS_Y}$. Therefore, there are associated with $n$
infinitesimals deformations $X_A$ of $X$ (induced by $n_1$) and
$Y_A$ of $Y$ (induced by $n_2$). Moreover, since $g(n)$ satisfies
the Maurer-Cartan equation in $M=KS_{Z}$, it defines an
infinitesimal deformation $Z _A$ of $ Z$. By construction, the
deformation $ Z _A$ is the product of the deformations $X_A$ and
$Y _A$.

\end{oss}

\medskip

Consider an infinitesimal deformation of the holomorphic map $f$
over $\Spec(A)$ as an infinitesimal deformation $ \Gamma_A$ of
$\Gamma$ over $\Spec(A)$ and $Z_A$ of $Z$ over $\Spec(A)$, with
$Z_A$ product of deformations of $X$ and  $Y$ over $\Spec(A)$.

By applying  Remark~\ref{oss (n,e^m) DEF(h,g)-> def T in def Xx
defY} and Theorem~\ref{teo def_k =Def_X}, the condition on the
deformation $Z_A$ is equivalent to requiring
$\shO_{Z_A}=\shO_{Z_A}(g(n))$, for some Maurer-Cartan element
$n\in  KS_X \times KS_Y$. Let $i^*:\shA_{Z}^{0,*} \lrg
\shA^{0,*}_\Gamma$ be the restriction morphism and let $\shI=\ker
i^* \cap \shO_{Z}$ be the holomorphic ideal sheaf of the graph
$\Gamma $ of $f$ in $Z$. The deformations $\Gamma_A$ of the graph
$\Gamma$ correspond to infinitesimal deformations $\shI_A\subset
\shO_{Z_A} $ of $\shI$ over $\Spec(A)$, with $\shI_A$ ideal
sheaves of $\shO_{Z_A}$, flat over $A$ and such that
$\shI_A\otimes _A \C\cong \shI$.

In conclusion, to give an infinitesimal deformation of $f$ over
$\Spec(A)$ (an element in $\Def(f)(A)$), it is sufficient to give
an ideal sheaf $\shI_A\subset \shO_{Z_A}(g(n))$ (for some $n\in
\MC_{KS_X \times KS_Y}$)  with $\shI_A$ $A$-flat and
$\shI_A\otimes _A \C\cong \shI$.

\begin{teo}\label{teo Def_(h,g)  Def (f)}
Let $h,g$ and $i^*$ be   as above. Then, there exists an
isomorphism of functors
$$
\gamma: \Def_{(h,g)}\lrg \Def (f).
$$
Given a local Artinian $\mathbb{C}$-algebra $A$ and  an element
$(n,e^m) \in \MC_{(h,g)}(A)$, we define a deformation of $f$ over
$\Spec(A)$ as a deformation $\shI_A(n,e^m)$ of the holomorphic
ideal sheaf of the graph  of $f$ in the following way
$$
\gamma(n,e^m)=\shI_A(n,e^m):=(\ker(\shA_{Z}^{0,0}\otimes A
\stackrel{\debar+\bl_{g(n)}}{\lrg} \shA^{0,1}_{Z}\otimes A))\cap
e^m (\ker i^* \otimes A) 
$$
$$
=\shO_{Z_A}(g(n))\cap e^m(\ker i^* \otimes A),
$$
where $\shO_{Z_A}(g(n))$ is the infinitesimal deformation of $Z$,
given by Theorem~\ref{teo def_k =Def_X},  that corresponds to
$g(n)\in MC_{KS_{X \times Y}}$.
\end{teo}

\begin{proof}
For each $(n,e^m) \in \MC_{(h,g)}(A)$ we have defined
$$
\shI_A(n,e^m) =\shO_{Z_A}(g(n))\cap e^m(\ker i^* \otimes A).
$$
First of all, we verify that this sheaf $\shI_A(n,e^m)\subset
\shO_{Z_A}(g(n))$ defines an infinitesimal  deformation of $f$;
therefore, we need to prove that it is flat over $A$ and
$\shI_A(n,e^m)\otimes _A \C\cong \shI$. It is equivalent to verify
these properties for $e^{-m}\shI_A(n,e^m)$. Applying
Lemma~\ref{lemma e^a(d+l_x)e^-a = d+e^a *x}, yields
$$
e^{-m}(\shO_{Z_A}(g(n)))=\ker(\debar + e^{-m} *g(n):
\shA_Z^{0,0}\otimes A \lrg \shA_Z^{0,1}\otimes A)
$$
and also
$$
e^{-m}\shI_A(n,e^m)=e^{-m}(\shO_{Z_A}(g(n)))\cap (\ker i^* \otimes
A)=
$$
$$
=\ker(\debar +e^{-m}*g(n))\cap(\ker i^* \otimes A).
$$
Since flatness is a local property, we can assume that $Z$ is a
Stein manifold. Then $H^1(Z,\Theta_Z)=0$ and $H^0(Z,\Theta_Z) \lrg
H^0(Z,N_{\Gamma|Z})$ is surjective. Since the following sequence
is exact
$$
\cdots \lrg H^0(Z,\Theta_Z) \lrg H^0(Z,N_{\Gamma|Z}) \lrg H^1(Z,L)
\lrg H^1(Z,\Theta_Z)\lrg \cdots,
$$
we conclude that  $H^1(L)=0$ or, equivalently, that the tangent
space of the functor $\Def_{L }$ is trivial. Therefore, $\Def_{L
}$ is the trivial functor.

This implies the existence of $\nu\in L^0 \otimes m_A$ such that
$e^{-m}*g(n)=e^\nu*0$ (by hypothesis, $e^{-m}*g(n)$ is a solution
of  the Maurer-Cartan equation  in $L$). Moreover, we recall that
$e^a(\ker  i^* \otimes A)=\ker i^* \otimes A$, for each $a \in
L^0\otimes m_A$.

Therefore,
$$
e^{-m}\shI_A(n,e^m)=\ker(\debar +e^{\nu}*0)\cap(\ker i^* \otimes
A)=\shO_{Z_A}(e^\nu *0)\cap (\ker i^* \otimes A)
$$
$$
=e^\nu(\shO_{Z_A}(0))\cap e^\nu(\ker i^* \otimes
A)=e^\nu(\shI\otimes A).
$$
Thus, $\shI_A(n,e^m)$  defines a deformation of $f$ and the
morphism
$$
\gamma: \MC_{(h,g)}\lrg \Def (f)
$$
is well defined, such that
$$
\gamma(A): \MC_{(h,g)}(A)\lrg \Def (f)(A)
$$
$$
(n,e^m) \longmapsto \gamma(n,e^m)=\shI_A(n,e^m).
$$

Moreover, $\gamma$ is  well defined on
$\Def_{(h,g)}(A)=\MC_{(h,g)}(A)/gauge$. Actually, for each $a \in
L^0 \otimes m_A$ and $b \in N^0\otimes m_A$,  we have
$$
\gamma(e^{b}*n,e^{g(b)}e^m e^{a})= \shO_{Z_A}(e^{g(b)}*g(n)) \cap
e^{g(b)}e^m e^{a}(\ker i^* \otimes A)=
$$
$$
e^{g(b)}\shO_{Z_A}(g(n)) \cap e^{g(b)}e^m (\ker i^* \otimes
A)=e^{g(b)}\gamma(n,e^m).
$$
This implies that the deformations $\gamma(n,e^m)$ and $\gamma(
e^b*n,e^{g(b)}e^m e^{a})$ are isomorphic (Remark~\ref{oss defo
mapp=Def grafo prod defX.defY}).

In conclusion, $\gamma: \Def_{(h,g)} \lrg \Def(f)$ is a well
defined natural transformation of functors.

In order to prove that $\gamma$ is an isomorphism   it is
sufficient to prove that
\begin{itemize}
  \item[i)] $\gamma$ is injective;
  \item[ii)] $\gamma$ induces a bijective map on the tangent
  spaces;
  \item[iii)] $\gamma$ induces an injective map on the
  obstruction spaces.
\end{itemize}

\medskip

i)  $\gamma\  is\  injective$. Suppose that $\gamma(n,e^m)=
\gamma(r,e^s)$,  then we want to prove that $(n,e^m)$ is gauge
equivalent to $ (r,e^s)$, i.e., there exist $a \in L^0 \otimes
m_A$ and $b \in N^0 \otimes m_A$ such that $e^b*r=n$ and
$e^{g(b)}e^se^a=e^m$.

By hypothesis, $\gamma(n,e^m)$ and $ \gamma(r,e^s)$ are isomorphic
deformations; then, in particular, the deformations induced on $Z$
are isomorphic. This implies that there exists $b \in N^0 \otimes
m_A$ such that $e^b*r=n$ and   $e^{g(b)}(\shO_{Z_A}(g(r)) )
=\shO_{Z_A}(g(n))$. Up to substituting $(r,e^s)$ with its
equivalent $(e^b*r,e^{g(b)}e^s)$, we can assume   to be  in the
following situation
$$
\shO_{Z_A}(g(n))\cap e^m(\ker i^* \otimes A)=\shO_{Z_A}(g(n))\cap
e^{m'} (\ker i^* \otimes A).
$$
Let $e^{a}=e^{-m'}e^{m}$, then
$$
e^{a}(e^{-m}(\shO_{Z_A}(g(n)))\cap (\ker i^* \otimes
A))=e^{-m'}(\shO_{Z_A}(g(n)))\cap (\ker i^* \otimes A).
$$
In particular, $e^a(e^{-m} (\shO_{Z_A}(g(n)))\cap (\ker i^*
\otimes A)) \subseteq \ker i^* \otimes A$.

Next, we prove, by induction, that $a \in L^0 \otimes m_A$ (thus
$e^m= e^{m'}e^a=e^{g(b)}e^se^a$).

Let $z_1,\ldots,z_n$ be holomorphic coordinates on $Z$ such that
$Z\supset \Gamma=\{z_{t+1}=\cdots=z_n=0\}$. Consider the
projection on the residue field
$$
e^{-m}(\shO_{Z_A}(g(n)))\cap (\ker i^* \otimes A )\lrg \shO_Z \cap
\ker i^*.
$$
Then, $z_i\in \ker i^*\cap \shO_Z$, for $i >t$. Since
$e^{-m}(\shO_{Z_A}(g(n)))\cap (\ker i^* \otimes A )$ is flat over
$A$, we can lift $z_i$ to $\tilde{z}_i=z_i + \varphi_i\in
e^{-m}(\shO_{Z_A}(g(n)))\cap (\ker i^* \otimes A ) $, with
$\varphi_i \in \ker i^* \otimes m_A$. By hypothesis,
\begin{equation}\label{equa e^a(ztilde)=e^a(z)+e^a(phi)}
e^a(\tilde{z}_i)=e^a(z_i)+e^a(\varphi_i) \in \ker i^*\otimes A.
\end{equation}
By Lemma~\ref{lemm L manetti is DGLA}, in order to prove that
$a\in L^0 \otimes m_A$ it is sufficient to verify that
$e^a(z_i)\in \ker i^*\otimes A $ and so, by (\ref{equa
e^a(ztilde)=e^a(z)+e^a(phi)}), that $e^a(\varphi_i) \in \ker i^*
\otimes A $.

If $A=\C[\epsi]$, then $\varphi_i \in \ker i^* \otimes \C\epsi$
and $ a \in \shA^{0,0}_Z \otimes \C\epsi$; this implies
$e^a(\varphi_i)=\varphi_i \in \ker i^* \otimes \C\epsi$.

Next, let $0 \lrg J \lrg {\tilde{A}} \stackrel{\alpha}{\lrg} A\lrg
0$ be a small extension. By hypothesis, $\alpha(a) \in L^0 \otimes
m_A$, that is, $\displaystyle \alpha(a)=\sum_{j=1}^n
\overline{a}_j \frac{\de}{\de z_j}$ with $\overline{a}_j \in \ker
i^* \otimes m_A$ for $j >t$. Let $a'_j$ be liftings of
$\overline{a}_j$. Then $a'_j \in \ker i^* \otimes m_{\tilde{A}}$,
for $j >t$ , $\displaystyle a'= \sum_{j=1}^n a'_j \frac{\de}{\de
z_j} \in L^0 \otimes m_{\tilde{A}} $ and $e^{a'}(\varphi_i)\in
\ker i^* \otimes m_{\tilde{A}}$. Since $\alpha(a)=\alpha(a')$,
then $a=a'+j$ with $j \in M^0 \otimes J$. This implies that
$e^a(\varphi_i) =e^{a'+j} (\varphi_i)= e^{a'}(\varphi_i) \in \ker
i^* \otimes m_{\tilde{A}}$.

\bigskip

As to $ii)$ and $iii)$, a straightforward computation shows that
the maps induced by $\gamma$ on tangent and obstruction spaces are
the isomorphisms induced by $\beta$ of Lemma~\ref{lem isomor
complesso Namba con Cono}.

\end{proof}

\begin{oss}
Consider the diagram
\begin{center}
$\xymatrix{  & & L \ar[d]^h   & \\
KS_{X} \times  KS_Y \ar[rr]^{g}  \ar@/^/[rrrd]_{\pi \circ g}& &
KS_{X \times Y}\ar[dr]^\pi &    \\
& & & A_X^{0,*}(f^*T_Y). &
 \\ }$
\end{center}

\noindent Since $h$ is injective, Lemma~\ref{lem h inj then C_(h,g)=C_pi
dopo g} implies the existence of    a quasi-isomorphism of
complexes $(\cil^\cdot,D)$ and $(C^\cdot_{\pi \circ
g},\check{\delta})$.

Then, we get the following exact sequence
\begin{equation}\label{succ esatta mia con ro uguale ziv}
\cdots \lrg H^1(C^\cdot_{\pi \circ g} )\stackrel{\varrho^1}{\lrg}
H^1(X,\Theta_X)\oplus H^1(Y,\Theta_Y)\lrg H^1(X,f^*\Theta_Y)\lrg
\end{equation}
$$
\qquad \lrg H^2(C^\cdot_{\pi \circ g} )\stackrel{\varrho^2}{\lrg}
H^2(X,\Theta_X)\oplus H^2(Y,\Theta_Y)\lrg H^2(X,f^*\Theta_Y)\lrg
\cdots ,
$$
where $\varrho^1$ and $\varrho^2$ are the projections on the
second factor and they are induced by the projection morphism
$\varrho:\Def(f) \lrg \Def_{KS_{X} \times  KS_Y }$ (see
Remark~\ref{oss proj varrho induce no EXT Def(h,g)->Def_N}).

In particular, $\varrho:\Def(f) \lrg \Def_{KS_{X} \times  KS_Y }$
associates with an infinitesimal deformation of $f$ the induced
infinitesimal deformations of $X$ and  $Y$.

Therefore, $\varrho^1$ associates with a first order deformation
of $f$ the induced first order deformations of $X$ and $Y$ and
$\varrho^2$ is a morphism of obstruction theories: the obstruction
to deform $f$ is mapped to the induced obstructions to deform $X$
and $Y$ (see also Remark~\ref{oss (n,e^m) DEF(h,g)-> def T in def
Xx defY}).

\end{oss}

\section{$L_{\infty}$ structure on the cone $\cil^\cdot$}
\label{sezio l infinito cono}

In this section we explicitly describe an $L_{\infty}$ structure
on the cone $\cil^\cdot$, associated with the pair of morphisms of
DGLAs $h:L \lrg M$ and $g:N \lrg M$. In particular, we prove  that
the deformation functor $\Def_{(h,g)}$ coincides with the functor
$\Def^\infty_{C_{(h,g)}^\cdot}$, associated with this $L_\infty$
structure on $\cil^\cdot$ (Theorem~\ref{teo Def INFI cil  =
Def(h,g)}). Finally, we show the existence of a DGLA $\hil$ that
controls the deformations of holomorphic maps (Corollary~\ref{cor
L infini Def_(h,g)iso Def_H(h,g)}).

\bigskip

First of all,  we briefly recall the definition of an $L_{\infty}$
structures on a graded vector space $V$. For a complete
description of such structures, see, for example, \cite{bib
fiorenz-Manet}, \cite{bib fukaya}, \cite{bib kontsevich} or
\cite[Chapter~IX]{bib manRENDICONTi}.

We denote by  $V[1]$  the complex $V$ with degrees shifted by $1$.
More precisely, for $\C[1]$ we have
$$
\C[1]^i=
  \begin{cases}
    \C & \text{if $i+1=0$}, \\
    0 & \text{otherwise}.
  \end{cases}
$$
Then, $V[1]=\C[1] \otimes V$, which implies $V[1]^i=V^{i+1}$.

Let  $\epsilon (\sigma; v_1,\ldots , v_n)$ be the Koszul sign. We
denote by $\bigodot^n V$ the space of co-invariant elements for
the action of $\Sigma_n$ on $\otimes^n V$ given by
$$
\sigma (v_1 \otimes \ldots \otimes  v_n)=\epsilon (\sigma;
v_1,\ldots , v_n)\ v_{\sigma(1)} \otimes \ldots \otimes
v_{\sigma(n)}.
$$
When $(v_1, \ldots , v_n)$ are clear by the context we simply
write $\epsilon (\sigma)$ instead of $\epsilon (\sigma; v_1,\ldots
, v_n)$.

\begin{dhef}
The set of {\it unshuffles } of type $(p,q)$ is the subset
$S(p,q)$ of $ \Sigma_n$ of permutations $\sigma$, such that
$\sigma(1)<\sigma(2) < \cdots < \sigma(p)$ and $\sigma(p+1)
<\sigma(p+2) < \cdots < \sigma(p+q)$.
\end{dhef}

\begin{dhef}
An $L_{\infty}$ structure  on a graded vector space $V$ is a
system $\{q_k\}_{k \geq 1}$ of linear maps $q_k \in Hom^1(\odot ^k
(V[1]),V[1])$ such that the map
$$
Q:\bigoplus_{n \geq 1} \bigodot^n V[1] \lrg \bigoplus_{n \geq 1}
\bigodot^n V[1],
$$
defined as
$$
Q(v_1 \odot \ldots \odot  v_n)=\sum_{k=1}^n \sum_{\sigma \in
S(k,n-k)} \epsi(\sigma)q_k( v_{\sigma(1)} \odot \ldots \odot
v_{\sigma(k)})\odot v_{\sigma(k+1)} \odot \ldots \odot
v_{\sigma(n)},
$$
is a co-derivation on the graded co-algebra $\bigoplus_{n \geq 1}
\bigodot^n V[1]$,  i.e.,  $Q \circ Q=0$.
\end{dhef}

\begin{oss}\label{oss DGLA is L infinito}
Let $(L,d,[\, , \, ])$ be a differential graded Lie algebra. Let
$q_1$ be the suspension of the differential $d$, i.e.,
$$
q_1:=d_{[1]}=Id_{\C[1]}\otimes d : V[1] \lrg V[1]
$$
$$
q_1(v_{[1]})=-(dv)_{[1]}.
$$
Then,  define   $q_2 \in Hom^1(\odot ^2 (V[1]),V[1])$  in the
following way
$$
q_2(v_{[1]}\odot w_{[1]})=(-1)^{deg_V v}[v,w]_{[1]}.
$$
Finally, defining $q_k=0$, for each $k \geq 3$, we endow the DGLA
$L$ with an $L_{\infty}$ structure, i.e., \emph{every DGLA is an
$L_{\infty}$-algebra with zero higher multiplications}.
\end{oss}

\begin{exe}\label{exe stru L infin su hil}
Consider the DGLA $M[t,dt]=M \otimes \C[t,dt]$, where $\C[t,dt]$
is the differential graded algebra of  polynomial differential
forms over the affine line. For every $a \in \C$, define the
\emph{evaluation morphism} in the following way
$$
e_a:M[t,dt] \lrg M,
$$
$$
e_a(\sum m_it^i +n_it^i dt)=\sum m_i a^i.
$$
The evaluation morphism is a morphism of DGLAs which is a left
inverse of the inclusion and   it is a surjective
quasi-isomorphism, for each $a$.

Next, define $K \subset L \times N \times M[t,dt] \times M[s,ds]$
as follows
$$
K=\{(l,n,m_1(t,dt),m_2(s,ds))|  \ h(l)=e_1(m_2(s,ds)),
g(n)=e_0(m_1(t,dt))  \}.
$$
$K$ is a DGLA with bracket and differential   defined as the
natural ones on  each component.

Consider the following morphisms of DGLAs:
$$
e_1: K \lrg M, \qquad (l,n,m_1(t,dt),m_2(s,ds))\longmapsto
e_1(m_1(t,dt))
$$
and
$$
e_0: K \lrg M, \qquad (l,n,m_1(t,dt),m_2(s,ds))\longmapsto
e_0(m_2(s,ds)).
$$

Let $H \subset K$  be defined as follow
$$
H=\{k \in K| \ e_1(k)= e_0(k) \},
$$
or written in  detail
$$
H=\{(l,n,m_1(t,dt),m_2(s,ds))\in L \times  N \times M[t,dt]\times
M[s,ds]\ |
$$
$$
h(l)=e_1(m_2(s,ds)), \ g(n)=e_0(m_1(t,dt)),\ e_1(m_1(t,dt))=
e_0(m_2(s,ds)) \}.
$$

\smallskip

\noindent For each $k=(l,n,m_1(t,dt),m_2(s,ds))\in K$, the pair $m_1(t,dt)$
and $m_2(s,ds)$ has fixed values at one of the extremes of the
unit interval. More precisely, the value of $ m_1(t,dt)$ is fixed
at the origin and $ m_2(s,ds)$ is fixed at 1. If $k$ also lies in
$H$, then there are conditions on the other extremes: the value of
$m_1(t,dt)$ at 1 has to coincide with the value of $m_2(s,ds)$ at
0.

Let
\begin{equation}\label{equa def H breve f=e_0 g=e_1}
\hil=
\end{equation}
$$
\{ (l,n,m(t,dt)) \in L\times N \times M[t,dt]\ | \,
h(l)=e_1(m(t,dt)), \ g(n)=e_0(m(t,dt)) \}.
$$
Since $e_i$ are morphisms of DGLAs, it is clear that $\hil$ is a
DGLA.

Moreover, considering the barycentric subdivision, we  get an
injective quasi-isomorphism
$$
\hil \hookrightarrow H,
$$
$$
(l,n,m(t,dt))\longmapsto (l,n,m\,(\, \displaystyle\frac{1}{2}\,
t,dt),m\,(\,\frac{s+1}{2},ds)).
$$

\smallskip

$\hil$ is the \emph{differential graded Lie algebra associated
with the pair $(h,g)$}.

\smallskip

Then, by Remark~\ref{oss DGLA is L infinito}, an $L_{\infty}$
structure on $\hil$ is defined by the following system  of linear
maps $q_k \in Hom^1(\odot ^k (\hil[1],\hil[1])$:
\begin{itemize}
  \item[-] $q_1(l,n,m(t,dt))=(-dl,-dn,-dm(t,dt))$;
  \item[-] $ q_2((l_1,n_1,m_1(t,dt))\odot (l_2,n_2,m_2(t,dt)))=$\\
  $(-1)^{deg_{\hil}(l_1,n_1,m_1(t,dt))}
  ([l_1,l_2],[n_1,n_2],[m_1(t,dt),m_2(t,dt)])$;
  \item[-] $q_k=0$, for every $k \geq 3$.
\end{itemize}
\end{exe}

A \emph{$L_{\infty}$-morphism} $f_{\infty}: (V,q_1,q_2,q_3, \ldots
)\lrg (W,p_1,p_2,p_3, \ldots )$ of $L_{\infty}$-algebras is a
sequence of degree zero linear maps
$$
f_n: \bigodot^n V[1] \lrg W[1], \qquad n\geq 1,
$$
such that the morphism of coalgebra
$$
F: \bigoplus_{n \geq 1} \bigodot^n V[1] \lrg \bigoplus_{n \geq 1}
\bigodot^n W[1],
$$
induced by $\sum_n f_n :\bigoplus_{n \geq 1} \bigodot^n V[1] \lrg
W[1]$, commutes with the codifferentials.

Sometimes this is the definition of a weak  $L_{\infty}$-morphism;
the strong  (or linear)  $L_{\infty}$-morphisms are the ones with
$f_n=0$, for each $n \geq 2$.

In particular, the \emph{linear part} $ f_1:V[1]\lrg W[1] $ of an
$L_{\infty}$-morphism $f_{\infty}: (V,q_1,q_2,q_3, \ldots )\lrg
(W,p_1,p_2,p_3, \ldots )$ satisfies the condition $f_1 \circ q_1
=p_1 \circ f_1$, i.e., $f_1$ is a map of differential complexes
$(V[1],q_1) \lrg (W[1], p_1)$.

A \emph{quasi-isomorphism} of $L_{\infty}$-algebra is an
$L_{\infty}$-morphism, whose linear part is a quasi-isomorphism of
differential complexes.

\medskip

The key result in this setting is the \emph{homotopical transfer
of   $L_\infty$ structures}.

\begin{teo}\label{teo trasfer strutt L infiniti}
Let $(V,q_1,q_2,q_3, \ldots)$ be an $L_\infty$-algebra and $(C,
\delta)$   a differential graded vector space. If there exist two
morphisms
$$
\pi:(V[1],q_1) \lrg (C[1],\delta_{[1]}), \qquad \imath:
(C[1],\delta_{[1]}) \lrg (V[1],q_1)
$$
which are homotopy inverses, then there exists an
$L_\infty$-algebra structure $(C,\langle  \ \rangle_1, \langle \
\rangle_2, \ldots )$ on $C$ extending its differential complex
structure, and making $ (V,q_1,q_2,q_3, \ldots)$ and $(C,\langle \
\rangle_1, \langle \ \rangle_2, \ldots )$ be quasi-isomorphic
$L_\infty$-algebra, via an $L_\infty$-quasi-isomorphism
$\imath_\infty$ extending $\imath$.
\end{teo}

\begin{proof}
See \cite[Theorem~4.1]{bib fiorenz-Manet}, \cite{bib fukaya} or
\cite{bib kontsevich}.
\end{proof}

Next, we use this theorem to transfer the $L_\infty$-structures of
$\hil$, given by  Example \ref{exe stru L infin su hil}, to the
cone $\cil ^\cdot$. We recall that  $ \cil^i=L^i \oplus N^i \oplus
M^{i-1}$ and $ D(l,n,m)=(d l, d n,-d m- g(n)+h(l)).$

Denote by $\langle \ ~\rangle_1\in
\Hom^1(\cil^\cdot[1],\cil^\cdot[1])$ the suspended differential,
i.e.,
$$
\langle   (l,n,m)\rangle_1=(-d l,-d n,d m +g(n)-h(l)).
$$

First of all, we note that  we can define an integral operator
$\int_a^b$ on   $M[t,dt]$, that is a linear map of degree -1 and
extends the usual integration $\int_a^b :\C[t,dt] \lrg \C$, i.e.,
$$
\int_a^b :M[t,dt]\lrg M \qquad \int_a^b (\sum_i t^im_i+t^idt\cdot
n_i)= \sum_i\left(\int_a^bt^idt\right)n_i.
$$

Next, define the following  linear  maps of degree 0:
$$
\imath: \cil^\cdot[1] \lrg \hil[1]
$$
$$
\imath(l,n,m)=(l,n,(1-t)g(n)+th(l)+dtm)\footnote{It is well
defined, since $e_0(l,n,(1-t)g(n)+th(l)+dtm)=g(n)$ and
$e_1(l,n,(1-t)g(n)+th(l)+dtm)=h(l)$.}
$$
and
$$
\pi: \hil[1]  \lrg \cil^\cdot[1]
$$
$$
\pi(l,n,m(t,dt))=(l,n,\int_0^1m(t,dt)).
$$
Finally, define the homotopy $K \in \Hom^{-1}(\hil[1],\hil[1])$ in
the following way
$$
K:\hil[1] \lrg \hil[1]
$$
$$
K(l,n,m(t,dt))=(0,0,t\int_0^1m(t,dt) -\int_0^t m(t,dt)).
$$

\begin{lem}\label{lemma i e pi quasi iso trasfer}
$\imath$ and $\pi$ are quasi-isomorphisms of complexes such that
$$
\pi \circ  \imath=id_{\cil^\cdot[1]} \quad \mbox{and}\quad
id_{\hil[1]}-\imath \circ \pi=K\circ q_1+q_1\circ K
$$
\end{lem}
\begin{proof}
See \cite[Lemma~3.2]{bib fiorenz-Manet}. It is a straightforward
computation.
\end{proof}

Thus, applying Theorem~\ref{teo trasfer strutt L infiniti}, we get
an $L_\infty$-algebra structure $\cil^\cdot$.

\begin{cor}\label{cor esist strutt L infin su Cil da transf Hil}
There exists an $L_\infty$-algebra structure $(\cil^\cdot, \langle
\ \rangle_1,\langle \ \rangle_2, \ldots )$ on the complex $\cil^\cdot$,
that extends its differential  structure, and it makes $
(\hil,q_1,q_2,0, \ldots)$ and $(\cil^\cdot,\langle  \ \rangle_1,
\langle \ \rangle_2, \ldots )$   quasi-isomorphic
$L_\infty$-algebras, via an $L_\infty$-quasi-isomorphism
$\imath_\infty$ extending $\imath$.
\end{cor}

\begin{oss}

As explained in \cite{bib fiorenza}, \cite{bib fiorenz-Manet},
\cite{bib fukaya} and \cite{bib kontsevich}, we have an explicit
description of the higher multiplication $\langle \ \rangle_n$ on
$\cil^\cdot$ in terms of rooted trees. Since
$$
q_2(\image K \otimes \image K)\subseteq \ker \pi \cap \ker K \quad
\mbox{and} \quad q_k=0,\ \forall \, k \geq 3,
$$
it can be proved that we have to consider  just the following
rooted trees
$$
\begin{xy}
,(-32,24)*{\circ};(-24,18)*{\bullet}**\dir{-}?>*\dir{>}
,(-32,12)*{\circ};(-24,18)**\dir{-}?>*\dir{>}
,(-24,6)*{\circ};(-16,12)*{\bullet}**\dir{-}?>*\dir{>}
,(-24,18)*{\circ};(-16,12)**\dir{-}?>*\dir{>}
,(-8,6)*{\circ};(0,0)*{\bullet}**\dir{-}?>*\dir{>}
,(-8,-6)*{\circ};(0,0)*{\bullet}**\dir{-}?>*\dir{>}
,(-16,12);(-8,6)*{\bullet}**\dir{.}?>*\dir{>}
,(-16,0)*{\circ};(-8,6)*{\bullet}**\dir{-}?>*\dir{>}
,(0,0)*{\bullet};(8,0)*{\circ}**\dir{-}?>*\dir{>}
,(15,0)*{{\text{root}}}
\end{xy}
$$
decorated with the operators $K,q_2,\imath$ and $\pi$ in the
following way
\begin{equation}\label{diagram albero decorato}
\begin{xy}
,(-32,24);(-24,18)*{\bullet}**\dir{-}?>*\dir{>}
,(-32,12);(-24,18)*{\bullet}**\dir{-}?>*\dir{>}
,(-24,6);(-16,12)*{\bullet}**\dir{-}?>*\dir{>}
,(-24,18)*{\bullet};(-16,12)*{\bullet}**\dir{.}?>*\dir{>}
,(-8,6);(0,0)*{\bullet}**\dir{-}?>*\dir{>}
,(-8,-6);(0,0)*{\bullet}**\dir{-}?>*\dir{>}
,(-16,12)*{\bullet};(-8,6)*{\bullet}**\dir{-}?>*\dir{>}
,(-16,0);(-8,6)*{\bullet}**\dir{-}?>*\dir{>}
,(0,0)*{\bullet};(8,0)**\dir{-}?>*\dir{>}
,(-7,8)*{\scriptstyle{q_2}} ,(-15,14)*{\scriptstyle{q_2}}
,(-23,20)*{\scriptstyle{q_2}} ,(-3,5)*{\scriptstyle{K}}
,(-11,11)*{\scriptstyle{K}} ,(-5,-2.4)*{\scriptstyle{\iota}}
,(-13,3.6)*{\scriptstyle{\iota}} ,(-21,9.6)*{\scriptstyle{\iota}}
,(-29,15.6)*{\scriptstyle{\iota}}
,(-29,23.3)*{\scriptstyle{\iota}} ,(4.5,1.5)*{\scriptstyle{\pi}}
,(1,2)*{\scriptstyle{q_2}},(11,0)*{.}
\end{xy}
\end{equation}

Explicitly, for  each  $n\geq 2$, these diagrams give us the
following formula for the higher multiplications $ \langle\
\rangle_n $:
\begin{multline*}
\langle\gamma_1\odot\cdots\odot
\gamma_n\rangle_n=\\
=\frac{(-1)^{n-2}}{2}\sum_{\sigma\in
\Sigma_n}\varepsilon(\sigma)\pi{q}_2(\imath(\gamma_{\sigma(1)})\odot
Kq_2(\imath(\gamma_{\sigma(2)})\odot\cdots\odot
Kq_2(\imath(\gamma_{\sigma(n-1)})\odot\imath
(\gamma_{\sigma(n)}))\cdots)).
\end{multline*}
The factor 1/2 in the above formula accounts for the cardinality
of the automorphisms group of the graph involved.

\end{oss}
In particular,   for each $ (l_1,n_1,m_1 )$ and $(l_2,n_2,m_2 )$
in $\cil^\cdot$ we have
$$
\langle(l_1,n_1,m_1 )\odot(l_2,n_2,m_2 ) \rangle_2=
$$
$$
\pi q_2(\imath( l_1,n_1,m_1  )\odot \imath(l_2,n_2,m_2 ))=
$$
$$
(-1)^{deg_{\cil ^\cdot  }(l_1,n_1,m_1 )}([l_1,l_2],[n_1,n_2],
$$
$$
([g(n_1),m_2] +[m_1,g(n_2)]) \int_0^1 \!\! (1-t)dt +([h(l_1),m_2]
+[m_1,h(l_2)])\int_0^1 \!\! tdt)=
$$
$$
(-1)^{deg_{\cil^\cdot }(l_1,n_1,m_1 )}
$$
$$
([l_1,l_2], [n_1,n_2], \frac{1}{2}\left( [g(n_1),m_2]
+[m_1,g(n_2)] + [h(l_1),m_2] + [m_1,h(l_2)] \right)).
$$
If  $(l_1,n_1,m_1 )= (l_2,n_2,m_2 ) =(l,n,m )\in \cil^0[1] $, then
$$
\langle  (l ,n ,m  )^{\odot 2} \rangle_2=(-[l,l],-[n,n],
-[m,g(n)]-[m,h(l)]\,).
$$

\begin{oss}
All higher multiplications (for $n\geq3$) vanish except the
following ones:
\begin{equation}\label{equa prodotti non nulli m^j l e m^j n}
\langle m_1\odot\cdots\odot m_{j}\odot l\rangle_{j+1} \qquad
\mbox{and } \qquad \langle m_1\odot\cdots\odot m_{j}\odot
n\rangle_{j+1}
\end{equation}
or their linear combinations (for each $j \geq 2$). Here we use
the notation $\gamma_i=m_i$ instead of $\gamma_i=(0,0,m_i)$ and
analogously for $\gamma_i=l$ or $\gamma_i=n$.
\end{oss}

As in \cite[Section~5]{bib fiorenz-Manet}, we can use Bernoulli's
numbers to give an explicit description of the multiplications of
(\ref{equa prodotti non nulli m^j l e m^j n}). First of all, we
recall that the  Bernoulli's numbers $B_n$ are defined by
$$
\sum_{n=0}^{\infty}B_n\frac{x^n}{n!}=\frac{x}{e^x-1}=
=1-\frac{x}{2}+\frac{x^2}{12}-\frac{x^4}{720}+\frac{x^6}{30240}-
\frac{x^8}{1209600}+\frac{x^{10}}{47900160}+\cdots.
$$
Next, consider the multiplication
$$
q_2( \imath(m) \odot \imath(l))=q_2((0,0,dt m)\odot (l,0,th(l))) 
$$
$$
= (-1)^{deg_{\hil}(0,0,dtm)}(0,0,tdt[m,h(l)]).
$$
Define recursively $\phi_j(t) \in \mathbb{Q}[t]$ and $I_j \in
\mathbb{Q}$ as
$$
 \phi_1(t)=t,\qquad I_j=\int_0^1\phi_j(t)dt,\qquad
\phi_{j+1}(t)= \int_0^t\phi_{j}(s)ds -tI_j.
$$
Then,
$$
K((\phi_j(t)dt)m)=-\phi_{j+1}(t)m
$$
and so
$$
Kq_2((dt m_1)\odot \phi_j(t)m_2)=
$$
$$
-(-1)^{deg_M m_1 +1}\phi_{j+1}(t)[m_1,m_2]=(-1)^{deg_M
m_1}\phi_{j+1}(t)[m_1,m_2].
$$

\begin{lem}\label{lemma prodotti cono <m m m m l>}
The following formula holds
$$
\langle m_1\odot \cdots \odot m_j\odot l \rangle_{j+1}=
$$
$$
(-1)^{n+\sum_{i=1}^j  deg_M(m_i)}I_j \sum_{\sigma\in
\Sigma_j}\varepsilon(\sigma) [m_{\sigma(1)}, [m_{\sigma(2)},\cdots
[m_{\sigma(j)},h(l)]\cdots]]
$$
$$
=-(-1)^{\sum_{i=1}^j   deg_M(m_i)}\frac{B_j}{j!}\sum_{\sigma\in
\Sigma_j}\varepsilon(\sigma) [m_{\sigma(1)}, [m_{\sigma(2)},\cdots
[m_{\sigma(j)},h(l)]\cdots]].
$$

\end{lem}

\begin{proof}
See \cite[Theorem~5.5]{bib fiorenz-Manet}.
\end{proof}

In particular, if $m_{\sigma(i)}=m \in M^0 $, for each $i$, then
$$
\langle m^{\odot j} \odot l \rangle_{j+1}
=-\frac{B_j}{j!}\sum_{\sigma\in \Sigma _j}  [m , [m ,\cdots [m
,h(l )]\cdots]]=-B_j\,\ad^j_m(h(l)).
$$

\bigskip

Next, consider the multiplication
$$
q_2( \imath(m) \odot g(n))=q_2((0,0,dt m)\odot (0,g(n),(1-t)g(n))=
$$
$$
(-1)^{deg_{\hil}(0,0,dtm)}(0,0,(1-t)dt[m,g(n)]).
$$
In this case, define recursively $\varphi_j(t) \in \mathbb{Q}[t]$
as follows
$$
 \varphi_1(t)=1-t=1- \phi_1(t)
$$
and
$$
\varphi_{j+1}(t)=\int_0^t\varphi_{j}(s)ds
-t\int_0^1\varphi_{j}(s)ds.
$$

We note that
$$
\varphi_2(t)=\int_0^t  \varphi_1(s)ds - t \int_0^1 \varphi_1(s)ds=
$$
$$
=\int_0^t(1- \phi_1(s))ds -t\int_0^1(1- \phi_1(s))ds=
$$
$$
( \int_0^t ds-t\int_0^1ds)- ( \int_0^t\phi_1(s)ds
-t\int_0^1\phi_1(s)ds)= -\phi_2(t).
$$
Therefore, for $j \geq 2$ we get
$$
\varphi_j(t)=-\phi_j(t).
$$
Then,
$$
K((\varphi_j(t)dt)m)=-\varphi_{j+1}(t)m
$$
and so
$$
Kq_2((dt m_1)\odot \varphi_j(t)m_2)=
$$
$$
-(-1)^{deg_M m_1 +1}\varphi_{j+1}(t)[m_1,m_2]=(-1)^{deg_M
m_1}\varphi_{j+1}(t)[m_1,m_2].
$$

\begin{lem}\label{lemma prodotti cono <m m m m n>}
The following formula holds
$$
\langle m_1\odot \cdots \odot m_j\odot n \rangle_{j+1}=
$$
$$
-(-1)^{n+\sum_{i=1}^j deg_M(m_i)}I_j \sum_{\sigma\in
S_j}\varepsilon(\sigma) [m_{\sigma(1)}, [m_{\sigma(2)},\cdots
[m_{\sigma(j)},g(n)]\cdots]]
$$
$$
= (-1)^{\sum_{i=1}^j  deg_M(m_i)}\frac{B_j}{j!}\sum_{\sigma\in
S_j}\varepsilon(\sigma) [m_{\sigma(1)}, [m_{\sigma(2)},\cdots
[m_{\sigma(j)},g(n)]\cdots]].
$$
\end{lem}

\begin{proof}
Analogous of Lemma~\ref{lemma prodotti cono <m m m m l>}.
\end{proof}

In particular, if $m_{\sigma(i)}=m \in M^0 $ for each $\imath$,
then
$$
\langle m^{\odot j} \odot n \rangle_{j+1}=
\frac{B_j}{j!}\sum_{\sigma\in S_j}  [m , [m ,\cdots [m
,g(n)]\cdots]]=B_j\,\ad^j_m(g(n)).
$$

\subsection{The Maurer-Cartan functor on the cone}

Let $T$ be an $L_\infty$-algebra. Then we can define (see
\cite{bib fukaya},\cite{bib kontsevich}) the \emph{Maurer-Cartan
functor $\MC_T^\infty$ associated with $T$} in the following way
$$
\MC_T^\infty : \Art \lrg \Set
$$
$$
\MC_T^\infty(A)=\left\{ \gamma\in T^1\otimes m_A
\;\strut\left\vert\; \sum_{j\ge1}\frac{\langle\gamma^{\odot
j}\rangle_j}{j!}=0\right.\right\}.
$$

Next, we want to prove that the Maurer-Cartan functor $\mc$
associated with the pair of morphisms of DGLAs $h:L \lrg M$ and
$g:N \lrg M$ (introduced in Section~\ref{sezi def MC_(h,g) and
DEF(h,g) no ext}) is exactly the Maurer-Cartan functor associated
with the $L_\infty$ structure on $\cil^\cdot$ defined before.

By definition,
$$
\MC_{\cil}^\infty:\Art \lrg \Set
$$
$$
\MC_{\cil}^\infty(A)=\left\{ \gamma\in \cil[1]^0\otimes m_A
\;\strut\left\vert\; \sum_{j\geq 1}\frac{\langle\gamma^{\odot
j}\rangle_j}{j!}=0\right.\right\}.
$$
Let $\gamma=(l,n,m) \in \cil[1]^0\otimes m_A$, thus $l \in L^1
\otimes m_A, \ n \in N^1 \otimes m_A$ and $ m \in M^0 \otimes
m_A$.
Then,
$$
\langle   (l,n,m )\rangle_1=(-dl,-dn,dm+g(n)-h(l))
$$
and
$$
\langle   (l,n,m )^{\odot 2}\rangle_2=(-[l,l],-[n,n],
-[m,g(n)]-[m,h(l)]).
$$

Therefore, the Maurer-Cartan equation
$$
\sum_{j\geq 1}\frac{\langle(l,n,m)^{\odot j}\rangle_j}{j!}=0
$$
splits into
$$
dl+\frac{1}{2}[l,l]=dn +\frac{1}{2}[n,n]=0
$$
and
\begin{equation}\label{equa second MC l infi cono}
g(n)-h(l)+dm -\frac{1}{2}[m,g(n)]- \frac{1}{2} [m,h(l)]
+\sum_{j\geq 3}\frac{\langle(l,n,m)^{\odot j}\rangle_j}{j!}=0.
\end{equation}
Since
$$
\sum_{j\geq 3}\frac{\langle(l,n,m)^{\odot j}\rangle_j}{j!}=
\sum_{j\geq 2} \frac{(j+1)}{(j+1)!}\langle m^{\odot j} \odot l
\rangle_{j+1} + \sum_{j\geq 2} \frac{(j+1)}{(j+1)!}\langle
m^{\odot j} \odot n \rangle_{j+1},
$$
applying  Lemma~\ref{lemma prodotti cono <m m m m l>} and
Lemma~\ref{lemma prodotti cono <m m m m n>}, Equation~(\ref{equa
second MC l infi cono}) becomes
$$
g(n)-h(l)+dm -\frac{1}{2}[m,g(n)]- \frac{1}{2} [m,h(l)]
$$
$$
+\sum_{j\geq 2} \frac{B_j}{j!} {\operatorname{ad}_m}^j(g(n))
-\sum_{j\geq 2} \frac{B_j}{j!} {\operatorname{ad}_m}^j(h(l))=0.
$$

Since $B_0=1$ and $B_1=-\frac{1}{2}$, we can write
$$
0=g(n)-h(l)+dm-[m,h(l)]+\sum_{j\geq 1} \frac{B_j}{j!}
{\operatorname{ad}_m}^j(g(n)) -\sum_{j\geq 1} \frac{B_j}{j!}
{\operatorname{ad}_m}^j(h(l))
$$
and so
$$
0=dm-[m,h(l))]+\sum_{j\geq 0} \frac{B_j}{j!}
{\operatorname{ad}_m}^j(g(n)) -\sum_{j\geq 0}
 \frac{B_j}{j!}  {\operatorname{ad}_m}^j(h(l))=
$$
$$
dm-[m,h(l))]+\sum_{j\geq 0} \frac{B_j}{j!}
{\operatorname{ad}_m}^j(g(n)  - h(l)).
$$

This implies that
$$
0=dm -[m,h(l)]+\dfrac{\ad_m}{e^{\ad_m}-id}(g(n)-h(l)).
$$
Applying the operator $\dfrac{e^{\ad_m}-id}{\ad_m}$,  we get
$$
g(n)=h(l)+\dfrac{e^{\ad_m}-id}{\ad_m}([m,h(l)]-dm)=e^m*h(l).
$$

In conclusion, the Maurer-Cartan equation for the $L_\infty$
structure on $\cil^\cdot$ is
$$
\begin{cases}
dl+\dfrac{1}{2}[l,l]=0\\
dn+\dfrac{1}{2}[n,n]=0\\
e^m * h(l)=g(n).
\end{cases}
$$
\begin{cor}\label{cor MC(h,g)=MC infini cil}
$\mc\cong \MC^\infty_{\cil^\cdot}$.
\end{cor}

\smallskip

Given an $L_\infty$-algebra $T$, two elements $x$ and $y
\in\MC_T^\infty(A)$ are \emph{homotopy equivalent} if there exists
$g[s,ds]\in \MC_{T[s,ds]}^\infty(A)$ with $g(0)=x$ and $g(1)=y$.
Then, the \emph{deformation functor} $\Def^\infty_T$ associated
with $T$ is $\MC_{T  }^\infty/\text{homotopy}$. Moreover, if two
$L_\infty$-algebras are quasi-isomorphic, then there exists an
isomorphism between the associated deformation functors \cite{bib
fukaya}. In particular, Corollary~\ref{cor esist strutt L infin su
Cil da transf Hil} implies the following result.

\begin{cor}\label{cor DEF inf cono iso DEF hil}
$\Def^\infty_{\cil^\cdot}\cong \Def_{\hil}$.
\end{cor}

Next, we prove that there is also an isomorphism between  the
functor $\Def_ {(h,g)} $ and the deformation functor $
\Def^\infty_{C_{(h,g)}^\cdot}$. First of all, we state two useful
lemmas. We will use the notation $e_a(p(t,dt)):=p(a)$, for each
$p(t,dt) \in M[t,dt]$ and $a \in \C$.

\begin{lem}\label{lemma MC_L[t,dt]=e^g(t)*x}
Let $M$ be a differential graded Lie algebra and let $A\in \Art$.
Then, for any $x$ in $\MC_M(A)$ and any $g(t)\in M^0[t] \otimes
m_A$, with $g(0)=0$, the element $e^{g(t)}  * x$ is an element of
$\MC_{M[t,dt]}(A)$. Moreover all the elements of
$\MC_{M[t,dt]}(A)$ are obtained in this way.
\end{lem}

\begin{proof}
See \cite[Corollary~7.2]{bib fiorenz-Manet}.
\end{proof}

\begin{lem}\label{lemma stab p^0 is irrelevant}
Let $x(t,dt) \in \MC_{M[t,dt]}(A)$, $\mu(t,dt)  \in
M[t,dt]^0\otimes m_A$, such that $\mu(0)=0$, and
$$
e^{\mu(t,dt) } *x(t,dt)=x(t,dt).
$$
Then, $e^{\mu(1)} \in Stab_A(x(1))$, i.e., there  exists $C\in
M^{-1}$ such that $\mu(1)=d_M C+ [x(1),C]$, where $d_M$ is the
differential in $M$.
\end{lem}

\begin{oss}
The proof of the case $x(t,dt)=0$ is contained in
\cite[Theorem~7.4]{bib fiorenz-Manet}.
\end{oss}

\begin{proof}
First, suppose that $x(t,dt)=x\in M$. Thus,  we have $e^{\mu(t,dt)
} *x=x$ and so
\begin{equation}\label{equaz lemma dmu+[x,mu]=0}
d(\mu(t,dt))+[x,\mu(t,dt)]=0\in (M[t,dt])^1 \otimes m_A.
\end{equation}
If we write $\mu(t,dt)=\mu^0(t) +\mu^{-1}(t) dt$, with $\mu^0(t)
\in M[t]^0$ and $\mu^{-1}(t) \in M[t]^{-1}$, then
Equation~(\ref{equaz lemma dmu+[x,mu]=0}) becomes
$$
\dot \mu^0(t)dt+d_M\mu^0(t)+d_M\mu^{-1}(t)dt  +[x, \mu^0(t)] +[x,
\mu^{-1}(t)]dt=0,
$$
or, equivalently,
$$
\begin{cases}
\dot{\mu}^0+d_M\mu^{-1}(t)+[x, \mu^{-1}(t)] =0\\
d_M\mu^0(t) +[x, \mu^0(t)]=0.
\end{cases}
$$
Thus, for any fixed $\mu^{-1}(t)$, we get
$$
\mu^0(t)=
$$
$$
-\int_0^t \,  d_M\mu^{-1}(s)ds -\int_0^t  \,
[x,\mu^{-1}(s)]ds=-d_M\int_0^t \,  \mu^{-1}(s)ds -[x,\int_0^t \,
\mu^{-1}(s)]ds.
$$
Let $C=-\int_0^1 \,  \mu^{-1}(s)ds \in M^{-1}$.  Therefore,
$\mu(1)=\mu^0(1)=d_M C+ [x,C]$ or, analogously, $e^{\mu(1)} \in
Stab_A(x)$. This concludes the proof in the case $x(t,dt)=x$.

Next, consider the general case of a  Maurer-Cartan element
$x(t,dt)\in \MC_{M[t,dt]}(A)$. Lemma~\ref{lemma
MC_L[t,dt]=e^g(t)*x} implies the existence of $g(t)\in M^0[t]$
such that $g(0)=0$ and
$$
x(t,dt)=e^{g(t)}*x(0).
$$
Therefore, the hypothesis  $ e^{\mu(t,dt)  } *x(t,dt)=x(t,dt)$,
implies
$$
e^{-g(t)} e^{\mu(t,dt) }e^{g(t)}*x(0)=x(0).
$$
Let $q(t,dt)= -g(t)\bullet \mu(t,dt)  \bullet g(t)$. If we write
$\mu(t,dt)=\mu^0(t) +\mu^{-1}(t)dt$ and $q(t,dt)=q^0(t)
+q^{-1}(t)dt$, then $q^0(t)=-g(t)\bullet \mu^0(t) \bullet g(t)$.

By the previous consideration  applied to $e^{q(t,dt)}*x_0=x_0$,
we conclude that $e^{q(1)} \in Stab_A(x(0))$.

The main property of irrelevant stabilizer asserts that
$$\forall\  a \in M^0\otimes A  \qquad e^a
Stab_A(x)e^{-a}=Stab_A(y), \quad \mbox{ with } \quad y=e^a*x.
$$
Therefore, $e^{\mu(1)}=e^{g(1)} e^{q(1)}e^{-g(1)}\in Stab_A(y)$,
with $y=e^{g(1)}*x(0)=x(1)$.

Equivalently, there exists $C\in M^{-1}$ such that $\mu(1)=d_M
 C+ [x(1),C]$.

\end{proof}

\begin{teo}\label{teo Def INFI cil  = Def(h,g)}
$$
\Def^\infty_{C_{(h,g)}^\cdot}=
\frac{\MC^\infty_{C_{(h,g)}^\cdot}}{\text{homotopy}} \simeq
\frac{\MC_{(h,g)}}{\text{gauge}}=\Def_ {(h,g)}.
$$
\end{teo}
\begin{proof}
This theorem is a generalization of \cite[Theorem~7.4]{bib
fiorenz-Manet}.

First, we show that \emph{gauge equivalence implies homotopy
equivalence}. Let $(l_0,n_0,m_0)$ and $(l_1,n_1,m_1)$  in
$\MC_{(h,g)}(A)$, for some $A \in \Art$; in particular,
$e^{m_0}*h(l_0)=g(n_0)$ and $e^{m_1}*h(l_1)=g(n_1)$.

Suppose that they are gauge equivalent elements, i.e., there exist
$a \in L^0 \otimes m_A$, $b \in N^0 \otimes m_A $ and $c \in
M^{-1}\otimes m_A $ such that
$$
l_1=e^a *l_0, \qquad n_1=e^b*n_0, \qquad
 m_1 = g(b) \bullet T \bullet m_0\bullet (-h(a)),
$$
with $ T=dc+[g(n_0),c]$ (and so $e^T \in Stab_A(g(n_0))$).

Let $\tilde{l}(s,ds)=e^{sa}*l_0\in L[s,ds]\otimes m_A$, $\tilde{n}
(s,ds) = e^{sb}*n_0\in N[s,ds]\otimes m_A$ and $T(s)=d(sc)+
[g(n_0),sc]$. By Lemma~\ref{lemma MC_L[t,dt]=e^g(t)*x},
$\tilde{l}$ and $\tilde{n}$ satisfy the Maurer-Cartan equation and
$h(\tilde{l})=e^{h(sa)}*h(l_0)$ and $g(\tilde{n})=
e^{g(sb)}*g(n_0)$.

Define $\tilde{m}=g(sb)\bullet T(s) \bullet m_0 \bullet (-h(sa))$;
then,
$$
e^{\tilde{m}}*h(\tilde{l}) =e^{g(sb)\bullet T(s) \bullet  m_0
\bullet (-h(sa))}*h(\tilde{l})=
$$
$$
e^{g(sb)\bullet T(s) \bullet m_0}*h(l_0)=e^{g(sb)\bullet T(s)}
* g(n_0)=g(\tilde{n}).
$$
Therefore, $(\tilde{l},\tilde{n},\tilde{m})\in
\MC^\infty_{\cil^\cdot[t,dt]}(A) $. Moreover, $\tilde{l}(0)=l_0$,
$\tilde{l}(1)=l_1$, $\tilde{n}(0)=n_0$, $\tilde{n}(1)=n_1$,
$\tilde{m}(0)=m_0$ and $\tilde{m}(1)=g(b)\bullet T \bullet
m_0\bullet (-h(a))=m_1$, i.e., $(l_0,n_0,m_0)$ and $(l_1,n_1,m_1)$
are homotopy equivalent.

Next, we show that \emph{homotopy equivalence implies  gauge
equivalence}. Let $(l_0,n_0,m_0)$ and $(l_1,n_1,m_1)$ be homotopy
equivalent elements in $\MC^\infty_{\cil^\cdot}(A)$. Thus, there
exists $(\tilde{l},\tilde{n},\tilde{m})\in
\MC^\infty_{\cil^\cdot[t,dt] }(A)$ such that
$$
d\tilde{l}+\dfrac{1}{2}[\tilde{l},\tilde{l}]=0, \qquad
d\tilde{n}+\dfrac{1}{2}[\tilde{n},\tilde{n}]=0, \qquad
g(\tilde{n})= e^{\tilde{m}}* h(\tilde{l})
$$
and
$$
\begin{cases}
(\tilde{l}(0),\tilde{n}(0),\tilde{m}(0))=(l_0,n_0,m_0)\\
(\tilde{l}(1),\tilde{n}(1),\tilde{m}(1))=(l_1,n_1,m_1).
\end{cases}
$$
In particular, $\tilde{l}$ and $\tilde{n}$ satisfy the
Mauer-Cartan equation  in $L[t,dt]$ and $N[t,dt]$, respectively.
Applying Lemma~\ref{lemma MC_L[t,dt]=e^g(t)*x}, there exist degree
zero elements $\lambda(t) \in L[t]\otimes m_A$ and $\nu(t) \in
N[t]\otimes m_A$,  such that $\lambda(0)=0$,
$\tilde{l}=e^{\lambda}* l_0$, $\nu(0)=0$ and $\tilde{n}=e^{\nu}* n
_0$.

This implies that $h(\tilde{l})=e^{h(\lambda)}* h(l_0)$,
$g(\tilde{n})=e^{g(\nu)}* g(n _0)$ and, for $s=1$,  that
$$
l_1=e^{\lambda(1)}* l_0 \qquad \mbox{  and } \qquad
n_1=e^{\nu(1)}*n_0.
$$
Moreover, we note that
$$
e^{g(\nu)  \bullet m_0\bullet (-h(\lambda))}*
h(\tilde{l})=e^{g(\nu)
   \bullet m_0}* h(l _0) =e^ {g(\nu)}* g(n _0)=
g(\tilde{n}) .
$$
Let $\mu=\tilde{m}\bullet h(\lambda)\bullet (-m_0) \bullet
(-g(\nu)) \in M^0[t,dt]\otimes m_A$ so that $\tilde{m}=\mu\bullet
g(\nu) \bullet m_0 \bullet (-h(\lambda))$. Then $\mu(0)=0$ and
$e^{\mu}*g(\tilde{n})=g(\tilde{n}) $.

Therefore, by  Lemma~\ref{lemma stab p^0 is irrelevant}, there
exists $C\in M^{-1}$, such that   $\mu(1)=  d_M(C)+[ g(n_1),C]$
and so  $e^{\mu(1)} \in Stab_A(g(n_1))$. Then, $m_1=
\tilde{m}(1)=\mu(1)\bullet g(\nu(1)) \bullet m_0\bullet
(-h(\lambda_1))$. Applying the  main property of the irrelevant
stabilizers, there exists  $C' \in M^{-1}$ such that
$$
\mu(1)\bullet g(\nu(1))=g(\nu(1))\bullet T',
$$
with $T'=dC' +[g(n_0),C']$ and   $e^{T'} \in Stab_A(g(n_0))$.
Thus,  $m_1=  g(\nu(1)) \bullet T' \bullet m_0\bullet
(-h(\lambda_1))$.

In conclusion,  if $(l_0,n_0,m_0)$ and $(l_1,n_1,m_1)$ are
homotopy equivalent, then there exists $(\lambda(1), \nu(1))\in
(L^0\otimes m_A)  \times (N^0\otimes m_A)$  and  $T'=dC'
+[g(n_0),C']  $, for some $C' \in M^{-1}$, such that
$$
\begin{cases}
l_1=e^{\lambda(1)} * l_0\\
n_1=e^{\nu(1)}*n_0 \\
m_1=  g(\nu(1))  \bullet T' \bullet  m_0 \bullet (-h(\lambda(1))),
\end{cases}
$$
i.e.,  $(l_0,n_0,m_0)$ and $(l_1,n_1,m_1)$ are gauge equivalent.

\end{proof}

\begin{cor}\label{cor L infini Def_(h,g)iso Def_H(h,g)}
$\Def_{(h,g)}\cong \Def_{\hil} $.
\end{cor}

Therefore, by suitably choosing  $L,M,$ and $h,g$, we have an
explicit description of the DGLA that controls the infinitesimal
deformations of holomorphic maps.

\begin{teo}\label{teo esiste hil governa DEF(f)}
Let $f:X \lrg Y$ be a  holomorphic map. Then, the DGLA $H_{(h,g)}$
associated with the morphisms $h:L \hookrightarrow KS_{X\times Y}$
and $g=(p^*,q^*):KS_X \times KS_Y \lrg KS_{X\times Y}$ (introduced
in Section~\ref{sezio def funct of pair morfisms}) controls
infinitesimal deformations of $f$, i.e.,
$$
\Def_{\hil} \cong \Def(f).
$$
\end{teo}
\begin{proof}
It is sufficient to apply  Theorem~\ref{teo Def_(h,g) Def (f)} and
Corollary~\ref{cor L infini Def_(h,g)iso Def_H(h,g)}.
\end{proof}

\end{document}